\documentclass[11pt]{amsart}

\usepackage{amsmath}
\usepackage{amsthm}
\usepackage{amssymb}
\usepackage{mathrsfs}
\usepackage{verbatim}

\usepackage{xcolor}
{%
\setcounter{enumi}{0}

\begin{enumerate}}%
{\end{enumerate} }

{%
\setcounter{enumi}{0}

\begin{enumerate}}%
{\end{enumerate} }

\usepackage[colorlinks]{hyperref}

\newtheorem{theorem}{Theorem}[section]
\newtheorem{lemma}[theorem]{Lemma}
\newtheorem{proposition}[theorem]{Proposition}
\newtheorem{corollary}[theorem]{Corollary}

\theoremstyle{definition}
\newtheorem{definition}[theorem]{Definition}
\newtheorem{remark}[theorem]{Remark}
\newtheorem{example}[theorem]{Example}

\numberwithin{equation}{section}

\allowdisplaybreaks

\newcommand*{\dist}{\operatorname{dist}}
\newcommand*{\lebm}{\mathcal{L}^n}
\newcommand*{\lebmone}{\mathcal{L}^1}
\newcommand*{\lebmtwo}{\mathcal{L}^2}
\newcommand*{\lebmthree}{\mathcal{L}^3}

\newcommand*\R{\mathbb{R}}
\newcommand*\N{\mathbb{N}}
\newcommand*\rn{\mathbb{R}^n}

\newcommand*\omegainfty{\Omega_{\infty}}

\newcommand*\re{\mathbb{R}}
\newcommand*\sph{\mathbb{S}}


\title[fractional Poincar\'e inequalities on unbounded domains]{Study of fractional Poincar\'e inequalities on unbounded domains}

\begin{document}


\maketitle

\centerline{\scshape Indranil Chowdhury$^1$ , Gyula Csat\'{o}$^2$ ,   Prosenjit Roy$^3$ and Firoj Sk$^3$}
\medskip 
{\footnotesize
  \centerline{1. Norwegian University of Science and Technology, Norway, indranil.chowdhury@ntnu.no }
  
 \centerline{2. Universitat de Barcelona, Spain, member of BGSMath Barcelona, part of the Catalan research}
 \centerline{group 2017 SGR 1392, supported
 by the MINECO grants MTM2017-83499-P and}
 \centerline{ ~MTM2017-84214-C2-1-P, and by the Mar\'ia de Maeztu Grant MDM-2014-0445, gyula.csato@ub.edu}
 \centerline{3. Indian Institute of Technology,  Kanpur, India, prosenjit@iitk.ac.in and firoj@iitk.ac.in}
  
}



\author{}
\address{}
\curraddr{}
\email{}
\thanks{}

\smallskip

\smallskip

\subjclass{\textit{Subject Classification}\ {26D10; 35R09; 46E35.}
\date{}


\maketitle

\begin{abstract}  The central aim of this paper is to study (regional) fractional Poincar\'e type inequalities on unbounded domains satisfying the finite ball condition.  Both existence and non existence type results on regional fractional inequality are established depending on various conditions on domains and on the range of $s \in (0,1)$. The best  constant in  both regional fractional and fractional Poincar\'e inequality is characterized for strip  like domains $(\omega \times \mathbb{R}^{n-1})$,  and the results obtained in this direction are analogous to those of the  local case. This settles one of the natural questions raised by  K. Yeressian in [\textit{Asymptotic behavior of elliptic nonlocal equations set in cylinders,  Asymptot. Anal. 89, (2014), no 1-2}].
\end{abstract}

\section{Introduction} \label{section:introduction}
By Poincar\'e inequality on a bounded domain $D\subset \rn$ for $1\leq p<\infty$ we get a constant $C=C(p,D)$ such that   
$$
    \int_{ D}|u|^p\;dx\leq C\int_{D}|\nabla u|^p\;
    dx,\;\;\text{ for all }u\in W^{1,p}_0(D),
$$
where the space $W^{1,p}_0(D)$ is the closure of the space   of smooth functions with compact support ($C_c^\infty(D)$)in the norm $||u||_{1,p,D}:=\big(\int_{ D}|u|^p\;dx+\int_{ D}|\nabla u|^p\;dx\big)^{1/p}$. This inequality remains true if $D$ is bounded in one direction and also if $D$ has finite measure. Also, it is easy to see that finite ball condition (see Definition \ref{def:finite ball cond}) is  necessary  for Poincar\'e inequality to hold.
A natural question is what are  the analogues of Poincar\'e inequalities in fractional Sobolev spaces, in particular for unbounded domains. 

Let $ \Omega$ be any open set in $\rn$, $0<s<1$  and let us define the Gagliardo semi-norm of $u$ as
$$
   [u]_{s,2,\Omega}:=\Big(\frac{C_{n,s}}{2}
   \int_{\Omega}\int_{\Omega}\frac{|u(x)-u(y)|^2}{|x-y|^{n+2s}}dxdy\Big)^\frac{1}{2}.
$$
Taking into account the constant $C_{n,s}$ will be crucial to study the best constants in fractional Poinca\'re inequality (c.f. Theorem \ref{Poincare 1}).  We refer to Section \ref{sec: known results} for details on the constant $C_{n,s}$.    Further, the fractional Sobolev Space $W^{s,2}(\Omega),$  is defined as
$$	
  W^{s,2}(\Omega):=\Bigg\{u\in L^2(\Omega):\int_{\Omega}
  \int_{\Omega}\frac{|u(x)-u(y)|  ^2}{|x-y|^{n+2s}}dxdy<\infty\Bigg\};
$$ 
endowed with the norm 
$$
    ||u||_{s,2,\Omega}:=\left(\|u\|_{L^2(\Omega)}^2+[u]_{s,2,\Omega}^2
    \right)^\frac{1}{2}.
$$
\noindent The spaces $W^{s,2}_0(\Omega)$ and $H^{s}_{\Omega}(\rn)$ denote the closure of $C_c^\infty(\Omega)$ with the norms $||\cdot||_{s,2,\Omega}$ and $\left(\|u\|_{L^2(\Omega)}^2+[u]_{s,2,\rn}^2\right)^\frac{1}{2}$ respectively. 
For more details about the fractional Sobolev space we refer to \cite{adams}, \cite{brasco}, \cite{valdi}, \cite{maz} and \cite{guide}. 
 These spaces play an important role in studying the Dirichlet problems involving fractional (and regional fractional) Laplace operators, see \cite{Chen18}, \cite{kass}, \cite{valdi} and \cite{raffela} for related works close to this direction.

\smallskip

To discuss the results regarding fractional Poincar\'e inequality, let us define 
$$
    P^1_{n,s}(\Omega):=\inf_{\substack{u\in W_0^{s,2}(\Omega) 
    \\ u\neq 0}}
    \frac{[u]_{s,2,\Omega}^2}{\displaystyle\int_{\Omega}u^2} 
    \;\text{ and }\; 
     P^2_{n,s}(\Omega):=\inf_{\substack{u\in H^{s}_{\Omega}(\rn) 
    \\ u\neq 0}}
    \frac{[u]_{s,2,\R^n}^2}{\displaystyle\int_{\Omega}u^2}.
$$
 If $P^1_{n,s}(\Omega)>0$, then we say \textit{regional fractional Poincar\'e} inequality  holds true. Whereas if $P^2_{n,s}(\Omega)>0$ we say \textit{fractional Poincar\'e} inequality  holds true.  
 $P^1_{n,s}$ and $P^2_{n,s}$ differ significantly in many properties, see for instance Proposition \ref{prop:elementary properties}, where we have summarized the known results on bounded domains. One of the relevant differences concerns the domain monotonicity: $P^2_{n,s}(\Omega_1)\geq P^2_{n,s}(\Omega_2)$ if $\Omega_1\subset\Omega_2$. However no domain monotonicity is known for $P^1_{n,s},$ but they have in common that the finite ball condition is required for both $P_{n,s}^1(\Omega),P_{n,s}^2(\Omega)>0$ to hold. Another known difference is the behaviour with respect to Schwarz symmetrization $u\mapsto u^{\ast}$ and $\Omega\mapsto\Omega^{\ast}$: it is well know that $P^2_{n,s}(\Omega^{\ast})\leq P^2_{n,s}(\Omega),$ see for example Frank and Seiringer \cite{Frank Seiringer}. But it is not true for $P^1_{n,s},$ see \cite{Li Wang}.

To the best of our knowledge, the study of regional fractional Poincar\'e and fractional Poincar\'e inequalities for general domains, in particular unbounded ones, is  largely open. Some literature is only available on fractional Poincar\'e inequalities in bounded domains and with zero mean value condition, see for instance the papers  \cite{dyda ihna Vaha},  \cite{vaha22}, and references therein.

It is easy to see that---as in the local case---the finite ball condition (see Definition \ref{def:finite ball cond}) is necessary for fractional Poincar\'e inequality to hold, regional or not. In particular the inequality cannot hold in $\re^n$ (see however \cite{mouhot} for fractional Poincar\'e inequality on $\re^n$ with zero mean value condition and with positive integrable weights).
One of the simplest unbounded domain satisfying the finite ball condition is the strip $\Omega_\infty:=\re^{n-1}\times (-1,1)$. For the local case, it is well known that the best Poincar\'e constant of $\Omega_\infty$ is same as the best Poincar\'e constant of the cross section $(-1,1).$ The proof is elementary, as we illustrate in $2$ dimension: if $\mu_1(\Omega)=\inf\{\int_{\Omega}|\nabla u|^2:\,u\in C_c^{\infty}(\Omega),\, \|u\|_{L^2(\Omega)}=1\}$  then for any $v\in C_c^{\infty}(\Omega_{\infty})$ one has the estimate, using the Poincar\'e inequality in $(-1,1),$
\begin{equation}
   \label{eq:local case lower bound}
     \|v\|_{L^2(\omegainfty)}^2\leq \frac{1}{\mu_1((-1,1))}\int_{\R}\int_{-1}^1v_{x_2}^2
      \leq\frac{1}{\mu_1((-1,1))}\int_{\omegainfty}(v_{x_1}^2+v_{x_2}^2).
\end{equation}
This shows that $\mu_1(\omegainfty)\geq \mu_1((-1,1)).$ The inverse inequality $\mu_1(\omegainfty)\leq \mu_1((-1,1))$ is easily shown by taking a sequence of function $w_\ell(x)=\varphi(x_2)v_\ell(x_1),$ where $\{v_\ell\}_{\ell\in\mathbb{N}}\subset C_c^{\infty}(\R)$ is an approximation of the function identically equal to $1$ in $\R$ and $\varphi$  satisfies  $\int_{-1}^1|\varphi'|^2 = \mu_1((-1,1)) \int_{-1}^1 \varphi^2$. For precise proof we refer to \cite{arn}.  One of the goal in this article is to  obtain similar characterization for both  $P^1_{n,s}(\Omega_{\infty})$ and $P^2_{n,s}(\Omega_{\infty})$.   In \cite{karen}, the author proved   $P^2_{n,s}(\Omega_{\infty}) >0$  and asked  whether  $P^2_{n,s}(\Omega_{\infty}) = P^2_{1,s}((-1,1))$[see, \cite{karen} section 4.2]. Theorem \ref{best constant} answers this question, even  for more general domain. Whereas, our next theorem deals with  similar issue  for regional fractional Poincar\'e inequality. Note that, the proofs are much more involved in our case, as there is no analogy to the estimate $v_{x_1}^2\leq v_{x_1}^2+v_{x_2}^2$ of \eqref{eq:local case lower bound}.

\begin{theorem}\label{Poincare 1}
For $\omegainfty=\R^{n-1}\times(-1,1)\subset\R^n$ the following statements hold:
\begin{enumerate}
    \item $P^1_{n,s}(\omegainfty)=P^1_{1,s}((-1,1))=0$, if $0<s\leq\frac{1}{2}$.
    \item $P^1_{n,s}(\omegainfty)>0$, if $\frac{1}{2}<s<1$. More precisely: the best constant $P^1_{n,s}(\omegainfty)$ is equal to the best constant of the cross section of the strip $\omegainfty$, i.e. 
$$
    P^1_{n,s}(\omegainfty)=P^1_{1,s}((-1,1)).
$$
\end{enumerate}
\end{theorem}

Next we deal with the characterization for $P^2_{n,s}(\Omega_\infty)$. However, the main idea of the proof differs from Theorem \ref{Poincare 1}, and we are able to give a more general result, thanks to characterization of $P^2_{n,s}(\Omega)$ as the first Dirichlet eigenvalue of the fractional Laplace operator on $\Omega$.  In the class of simply connected domains  in two dimensions, it was shown in \cite{Sandeep-Man} that finite ball condition [see definition \ref{def:finite ball cond}] is a necessary and sufficient condition for local Poincar\'e inequality to hold true. In \cite{ip} it was established that such a result cannot hold for both $P_{2,s}^1$ and $P_{2,s}^2$  in the range $s\in (0,\frac{1}{2}).$

\begin{theorem}\label{best constant}
Consider the strip $\omegainfty=\re^m\times\omega$ in $\rn$ with $1\leq m<n$, where $\omega$ is a bounded open subset of $\re^{n-m}$. Then for $0<s<1$, we have 
$$
    P^2_{n,s}(\omegainfty)=P^2_{n-m,s}(\omega).
$$
\end{theorem}
\noindent The proof of the above theorem follows by two main steps. First, as an  application of discrete Picone identity we  prove  that $ P^2_{n,s}(\omegainfty) \geq P^2_{n-m,s}(\omega)$. The other inequality is obtained by  constructing suitable test functions on truncated domains $\Omega_\ell = B_{m}(0,\ell) \times \omega$ and  then finally letting $\ell$ tend to infinity.  Very recently, and after our result have appeared, an independent proof of the above theorem was also given by \cite{Ambrosio Freddi Musina}.
Independently, various kinds of problems (mainly PDEs) on $\Omega_\ell$ have been considered, and their asymptotic behavior as $\ell\to\infty$ is studied. Such kind of theories are now well studied  in the local case and for more details on this subject we refer \cite{chipot1}, \cite{delpino}, \cite{mojsic},  \cite{crs}, \cite{arn1}, \cite{arn}, \cite{karen} and the references therein.

\smallskip
The rest of the paper discusses the existence and non-existence issues regarding  regional fractional Poincar\'e inequality.  The lack of any known domain monotonicity property for $P^1_{n,s}(\Omega)$ makes the study of regional fractional Poincar\'e inequality more interesting, even for specific domains, or any special class of domains. Our next theorem provides
 a sufficient conditions on the domain $\Omega$ for which regional fractional Poincar\'e inequality remains true.  At the end of Section \ref{section:Proof 1.1 and 1.2} we will give some examples of  domains which satisfy the hypothesis of Theorem  \ref{Poincare gen domain 1}. Here $\mathbb{S}^{n-1}\subset\mathbb{R}^n$ denotes the unit sphere and $\mathcal{H}^{n-1}$ denotes the $(n-1)$-dimensional Hausdorff measure.

\begin{theorem}\label{Poincare gen domain 1}
Let $\Omega\subset\R^n$ be a measurable set and $\frac{1}{2}<s<1.$ Suppose there exists $\Sigma\subset\mathbb{S}^{n-1}$ with $\mathcal{H}^{n-1}(\Sigma)>0$ and  such that for all $w\in \Sigma$ and all $x\in\R^n$ the one dimensional intersections with $\Omega$
$$
   A_{x,w}:= \left\{ t \in \mathbb{R}: x+tw \in  \Omega\right\} \quad\text{ satisfy uniformly one dimensional  finite ball condition},
$$
that is
$$
    \sup\left\{ \operatorname{length}(I):\, I\text{ interval, } 
    I\subset  A_{x,w} \right\}=m <\infty.
$$
Then regional fractional Poincar\'e inequality holds, more precisely
$$
   P_{n,s}^1(\Omega)\geq \frac{C_{n,s}}{2 C_{1,s}}\mathcal{H}^{n-1}
   (\Sigma)\frac{P_{1,s}^1((0,1))}{m^{2s}}.
$$
\end{theorem}

\noindent An example of a domain which does not satisfy the hypothesis of Theorem \ref{Poincare gen domain 1} (but satisfies the finite ball condition) is an infinite union of concentric annuli, see Example \ref{exmpl:domains with sufficient condn} (v):
\begin{equation}
 \label{intro:concentric annuli}
  C=\bigcup_{k=1}^{\infty}B_{2k}(0)\setminus B_{2k-1}(0).
\end{equation}
\smallskip

On the other hand, we know $P^1_{n,s}(\Omega)=0$ if $\Omega$ is bounded and $s\in(0,\frac{1}{2})$ [see Proposition \ref{prop:elementary properties}]. It has been observed by Frank, Jin and  Xiong \cite{rupert} that the same proof works also for domains with finite measure that satisfies \eqref{eq:rupert measure condition}, [see Remark \ref{remark:bounded perimeter}]. 
For example, if we consider the domain 
\begin{equation}
  \label{example_not_rupert}
  \widetilde{\Omega}=\{(x_1,x_2)\in\R^2:\,x_1>1,\;0<x_2<1/x_1^{1+\epsilon}\}  \ \textrm{where} \ \epsilon > 0,
\end{equation}
then by \cite{rupert} it follows  $P^1_{n,s}(\widetilde{\Omega})=0$ whenever $0 < s \leq  \frac{\epsilon}{2(1+\epsilon)}$, but the same result is inconclusive for the range $\frac{\epsilon}{2(1+\epsilon)} < s \leq \frac{1}{2}$. However, in this context we can apply our next theorem  for the whole range of $s \in (0, \frac{1}{2})$ and conclude that $P^1_{n,s}(\widetilde{\Omega})=0$.  Noteworthy, using Theorem \ref{Poincare gen domain 1} to $\widetilde{\Omega}$, we see that $P^1_{n,s}(\widetilde{\Omega})>0$ for $s> \frac{1}{2}$. In our next theorem,  we prove  a very general condition on domains (not only of finite measure) for which regional fractional Poincar\'e inequality does not hold.



\begin{theorem}
\label{theorem:distance condition}
Let $0<s<\frac{1}{2}$ and suppose that $\Omega\subset\R^n$ is a Lipschitz set. Suppose there exists a bounded open Lipschitz set $U$ and a sequence $\lambda_{k}$ tending to infinity such that
$$
    \lim_{k\to\infty}\frac{1}{\mathcal{L}^n(\lambda_k U\cap \Omega)}
    \int_{\lambda_k U\cap \Omega}
    \frac{dx}{\dist(x,(\lambda_kU)^c)^{2s}}=0,
$$
where $\lambda_kU=\{\lambda_k x:\,x\in U\}.$ Then $P^1_{n,s}(\Omega)=0.$
\end{theorem}

\noindent  In section \ref{sec:sufficent_cndn_unbd_dom_0},  we discuss the application of this result in details by providing several scenarios through examples. 
For instance, the result applies to any domain which satisfies the growth condition $\mathcal{L}^n(\Omega\cap B_R)\geq c R^n$  [see, Corollary \ref{cor:domain condn}], the domain of type \eqref{intro:concentric annuli} is such an example. 

Whether $P^1_{n,s}(\Omega)=0$ for every unbounded domain if $s\in(0,\frac{1}{2})$ is still an open problem to the best of our knowledge. We end the  paper by a last result in this direction and give a sufficient condition for unbounded domains, but with finite measure, for  $P^1_{n,s}(\Omega)=0$ to hold. See Theorem \ref {theorem:general s smaller half in tube} in Section
\ref{Section:sufficient condition for domains with finite measure}. This theorem, vaguely speaking, deals with a wider class of domains which are of the type \eqref{example_not_rupert}.

\section{Some Elementary properties and known results} \label{sec: known results}
We briefly fix the notation that we will use throughout this paper. For an integer $n$ and a measurable set $\Omega\subset\R^n$ we write $\mathcal{L}^n$ to denote the Lebesgue measure, or shortly $|\Omega|$ if there is no ambiguity concerning $n.$  $B_R(x)$ denotes a ball of radius $R$ centered at $x.$ We omit the center if $x=0.$ We will use that
\begin{equation}
 \label{eq:surface of S n-1}
    \mathcal{H}^{n-1}\left(\mathbb{S}^{n-1}\right)=\frac{2\pi^{\frac{n}{2}}}
    {\Gamma\left(\frac{n}{2}\right)},\quad\text{where $\Gamma$ is standard Gamma 
    function}.
\end{equation}
We will use the Beta function, which is defined for $x,y>0$ by
$$
   B(x,y):=\int_0^1t^{x-1}(1-t)^{y-1}dt,
$$
and its properties:
\begin{equation}
 \label{eq:properties of Beta}
   B(x,y)=2\int_0^{\frac{\pi}{2}}(\sin\theta)^{2x-1}(\cos\theta)^{2y-1}d\theta,
   \quad
   B(x,y)=\frac{\Gamma(x)\Gamma(y)}{\Gamma(x+y)}.
\end{equation}

The constant $C_{n,s},$ as in introduction is uniquely defined if we want consistency with fractional partial integration and with the fractional Laplacian in the following sense
$$
   [u]_{s,2,\R^n}^2=\int_{\R^n}u(-\Delta_n)^su,
$$
and $(-\Delta_n)^su=\mathcal{F}^{-1}(|\xi|^{2s}(\mathcal{F}u)(\xi)),$ where $\mathcal{F}$ is the Fourier transform, see \cite{guide} for details. The fractional Laplace operator can be equivalently defined as
 $$
     (-\Delta_n)^su(x)=C_{n,s}\;P.V.\int_{\rn}\frac{u(x)-u(y)}{|x-y|^{n+2s}}dy.
 $$
Here P.V. denotes the principal value and the above integral is defined for $u\in C_c^2(\rn).$ We refer to \cite{pi},\cite{fr},\cite{guide},\cite{sar} and \cite{strict positive} for related work concerning fractional Laplace operator. The constant $C_{n,s}$ is explicitly given by
\begin{equation}
\label{const_flap}
   C_{n,s}=\left(\int_{\R^n}\frac{1-\cos(z_1)}{|z|^{n+2s}}dz\right)^{-1}=
   \frac{s2^{2s}\Gamma\big(\frac{n+2s}{2}\big)}{\pi^{\frac{n}{2}}\Gamma(1-s)}.
\end{equation}
We first list some simple and known properties regarding the $[\cdot]_{s,2,\Omega}$ norm and the fractional Poincar\'e inequality. The regional fractional Poincar\'e inequality on bounded domains is deduced from the fractional Hardy inequality, which we recall here, stated only for the case required here.

\begin{theorem}(Dyda \cite{lk}, see also \cite{dy})
\label{theorem:Dyda fractional Hardy}
Let $\Omega\subset\R^n$ be a bounded open set with Lipschitz boundary $\partial\Omega$ and $\frac{1}{2}<s.$ Then there exists a constant $C$ depending only on $\Omega,n,s$ such that
$$
   \int_{\Omega}\frac{|u(x)|^2}{(\dist(x,\partial\Omega))^{2s}}dx\leq    
   C(\Omega,n,s)\int_{\Omega}\int_{\Omega}
   \frac{|u(x)-u(y)|^2}{|x-y|^{n+2s}}dx\,dy\quad\text{for all }u\in C_c^{\infty}(\Omega).
$$
\end{theorem}

We mention that the  validity of  fractional Hardy type inequalities on unbounded domains has been characterized in \cite[Theorem 6.1]{vaha1}. Their characterization is in terms of a uniform inverse subadditivity inequality for capacitiy involving the intersection of any compact subset of $\Omega$ with a Whitney decomposition of $\Omega.$  Such a condition is however difficult to verify in practice.

\smallskip

We recall the following definition.

\begin{definition}[\textbf{Finite ball condition}]
\label{def:finite ball cond}
We say that a set $\Omega\subset\R^n$ satisfies the \textit{finite ball condition} if $\Omega$ does not contain arbitrarily large balls, that is
$$
   \sup\{ r:\, B_r(x)\subset\Omega,\; x\in \Omega\}<\infty.
$$
\end{definition}

The finite ball condition plays an important role in the characterization of domains for which the local Poincar\'e inequality holds. It is immediate to see (by domain monotonicity and scaling) that the finite ball condition is necessary for local Poincar\'e to hold. But it is actually equivalent for simply connected domains in dimension 2, see Sandeep-Mancini \cite{Sandeep-Man}.

\begin{proposition}
\label{prop:elementary properties} Let $n\geq 1$ be a positive integer and $0<s<1.$
\smallskip

(i) Let $\Omega\subset\R^n$ be a bounded open Lipschitz set. Then
$$
    P^1_{n,s}(\Omega)>0 \quad\text{ if } \frac{1}{2}<s<1,\quad\text{ and }\quad 
    P^1_{n,s}(\Omega)=0\quad\text{ if } 0<s \leq \frac{1}{2}.
$$
whereas (this does not require $\Omega$ to be Lipschitz)
$$
    P^2_{n,s}(\Omega)>0  \quad\text{ if } 0<s<1.
$$

(ii) Let $\Omega\subset\R^n,$ $t>0$ and $u\in W^{s,2}(\Omega).$ Define $v_t\in W^{s,2}(\Omega)$ by $v_t(x)=u(tx).$ Then
$
   [u]_{s,2,t\Omega}^2=t^{n-2s}[v_t]_{s,2,\Omega}^2,
$
and moreover
$$
    P^1_{n,s}(t\Omega)= \frac{P^1_{n,s}(\Omega)}{t^{2s}},\qquad 
    P^2_{n,s}(t\Omega)= \frac{P^2_{n,s}(\Omega)}{t^{2s}}
.$$

(iii) $P^2_{n,s}$ has the domain monotonicity property, i.e. if $\Omega_1\subset\Omega_2$ then $P^2_{n,s}(\Omega_2)\leq P^2_{n,s}(\Omega_1).$
\smallskip 

(iv) If $\Omega\subset\R^n$ does not satisfy the finite ball condition then
$
    P^1_{n,s}(\Omega)=P^2_{n,s}(\Omega)=0.
$
\end{proposition}

\begin{remark}
\label{remark:bounded perimeter}
The hypothesis $\Omega$ bounded and Lipschitz in part (i)  for $P^1_{n,s}(\Omega)=0$ in the case $0<s<\frac{1}{2}$ can be weakened significantly, see \cite{rupert} Lemma A.2. It is sufficient to require that $\Omega$ has finite measure and that
\begin{equation}
 \label{eq:rupert measure condition}
    \mathcal{L}^n(\{x\in\Omega:\,\dist(x,\Omega^c)<\delta\})=o(\delta^{2s}).
\end{equation}
For a bounded Lipschitz domain one has the estimate $\mathcal{L}^n(\{x\in\Omega:\, \dist(x,\Omega^c)<\delta\})\leq C \mathcal{H}^{n-1}(\partial \Omega) \delta.$ Such an estimate remains true if $\partial\Omega$ is only piecewise Lipschitz and the condition \eqref{eq:rupert measure condition} is satisfied as long as $\mathcal{H}^{n-1}(\partial \Omega)<\infty.$ We will need this observation to use the result for the intersection of two Lipschitz domains.
\end{remark}

\begin{remark}
\label{remark:finite ball cond in 1 dim}
It is immediate to see that in $1$ dimension it holds that
$$ 
    P^1_{1,s}(\Omega)>0\quad\Leftrightarrow\quad \Omega\text{ satisfies finite ball 
    condition}.
$$
The necessity of the condition is (iv) of the previous proposition. For the sufficiency, assume that $\Omega\subset\R^1$ is an open set satisfying the finite ball condition. Hence there exists a countable number of open intervals $I_k$ such that
$$
   \Omega=\bigcup_{k=1}^{\infty}I_k,\quad I_k\cap I_j=\emptyset\text{ if }k\neq j,\quad 
   \sup_k(\operatorname{length}(I_k))=m<\infty.
$$
Hence, using (ii), we get for $u\in C_c^{\infty}(\Omega)$ 
\begin{align*}
   [u]^2_{s,2,\Omega}\geq &\frac{C_{1,s}}{2}
    \sum_{k=1}^{\infty}\int_{I_k}\int_{I_k}\frac{|u(x)-u(y)|^2}
    {|x-y|^{1+2s}}dx\,dy
    \geq 
    \sum_{k=1}^{\infty} P^1_{1,s}(I_k)\int_{I_k}u^2
    \smallskip
    \\
    =& \sum_{k=1}^{\infty} 
    \frac{1}{(\operatorname{length}(I_k))^{2s}}P^1_{1,s}((0,1))\int_{I_k}u^2
    \geq \frac{P^1_{1,s}((0,1))}{m^{2s}}\int_{\Omega}u^2,
\end{align*}
which gives
\begin{equation}
  \label{eq:P11 estimated by m}
    P^1_{1,s}(\Omega)\geq \frac{P^1_{1,s}((0,1))}{m^{2s}}.
\end{equation}
\end{remark}

\begin{proof}[Proof of Proposition \ref{prop:elementary properties}.]
(i) \textit{Case  $\frac{1}{2}<s < 1$:}  
This follows from Theorem \ref{theorem:Dyda fractional Hardy} and the estimate $\dist(x,\partial\Omega)\leq \frac{1}{2} \operatorname{diam(\Omega)}$ for any $x\in\Omega.$ Hence
$$
   P^1_{n,s}(\Omega)\geq 
   \left(\frac{1}{2}\operatorname{diam}(\Omega)\right)^{-2s} C^{-1},
$$
where $C=C(\Omega,n,s)$ is the constant in Theorem \ref{theorem:Dyda fractional Hardy} .
\smallskip

\textit{Case $0<s \leq \frac{1}{2}$:} It is well known  [see, Theorem 11.1 in \cite{Lions-Magenes}] that $C_c^\infty(\Omega)$ is dense in $W^{s,2}(\Omega)$ if and only if $s \leq \frac{1}{2}.$ Notice  that the  constant function $1 \in W^{s,2}(\Omega)$ and $[1]_{s,2,\Omega} = 0$. Hence the claim follows.

However, we present a detailed proof for the case $s \in (0,\frac{1}{2})$, as it will be useful later. Let $T_{\delta}=\{x\in\Omega:\, \dist(x,\partial\Omega)\leq \delta\}$ be some small interior tubular neighborhood of $\partial\Omega$ for $\delta>0.$ $u_{\delta}$ shall be an approximation, as $\delta\to 0,$ of the characteristic function of $\Omega:$ 
\begin{equation}
 \label{eq:properties u delta}
    u_{\delta}\in C_c^{\infty}(\Omega),\quad 0\leq u_{\delta}\leq 1,\quad 
    u_{\delta}=1\text{ in }\Omega\setminus T_{\delta}, \quad 
    \frac{|u_{\delta}(x)-u_{\delta}(y)|}{|x-y|}\leq |\nabla u_{\delta}|\leq \frac{2}{\delta}\text{ in }T_{\delta}\,.
\end{equation}
Hence for $\delta\to 0$ one has $\|u_{\delta}\|_{L^2(\Omega)}\to |\Omega|$ and it is sufficient to show that $[u_{\delta}]_{s,2,\Omega}^2\to 0.$ As $|u_{\delta}(x)-u_{\delta}(y)|=0$ for $x,y\in \Omega\setminus T_{\delta}$ we get
$$
   [u_{\delta}]_{s,2,\Omega}^2\leq 2 \int_{T_{\delta}} dx\int_{\Omega}dy
   \frac{|u_{\delta}(x)-u_{\delta}(y)|^2}{|x-y|^{n+2s}}=A+B,
$$
where
\begin{align*}
  A=\int_{T_{\delta}} dx\int_{\{ y\in \Omega ; \  |x-y|<\delta \}}dy
   \frac{|u_{\delta}(x)-u_{\delta}(y)|^2}{|x-y|^{n+2s}}
   \textrm{ and }
   B=\int_{T_{\delta}} dx\int_{\{ y\in \Omega ; \  |x-y|<\delta \}}dy
   \frac{|u_{\delta}(x)-u_{\delta}(y)|^2}{|x-y|^{n+2s}}.
\end{align*}
For $A$ one uses the last property of \eqref{eq:properties u delta} and the estimate
$$
    \int_{\{ y\in \Omega ; \  |x-y|<\delta \}}dy
   \frac{1}{|x-y|^{n-2+2s}}
   \leq 
   \int_{ |x-y|<\delta}dy
   \frac{1}{|x-y|^{n-2+2s}}
   =\int_{B_{\delta}(0)} \frac{dy}{|y|^{n-2+2s}}.
$$
After radial integration and using that $|T_{\delta}|$ is of the order $\mathcal{H}^{n-1}(\partial\Omega)\delta$ one gets that $A\leq C\delta^{1-2s}$ for some constant $C=C(n,s,\Omega).$ For $B$ one uses the estimates $|u_{\delta}(x)-u_{\delta}(y)|\leq 1$ and
$$
   \int_{\{ y\in \Omega ; \  |x-y|<\delta \}}dy
   \frac{1}{|x-y|^{n+2s}}
   \leq
   \int_{ |x-y|>\delta}dy
   \frac{1}{|x-y|^{n+2s}}
   =\int_{B_{\delta}(x)^c}\frac{1}{|x-y|^{n+2s}}dy.
$$
After radial integration, one proceeds as for estimating $A,$ and concludes that $B$ is also of order $\delta^{1-2s}.$
\smallskip

The last statement concerning $P_{s,2}^2(\Omega)$ follows from the fractional Sobolev embedding in $\R^n$ and the fact that $2\leq 2^{\ast}:=2n/(n-2s),$ and that $L^{2^{\ast}}$ is continuously embedded into $L^2$ if $\Omega$ is bounded.
\smallskip

(ii) This is immediate by change of variables.
\smallskip

(iii) Follows directly from the definition.

\smallskip

(iv) Let $x\in\Omega$  and $r>0$ be such that $B_r(x)\subset\Omega.$ Then by (iii) and (ii) we get that
$$
   P^2_{n,s}(\Omega)\leq P^2_{n,s}(B_r(x))=P^2_{n,s}(B_r(0))=\frac{P^2_{n,s}(B_1(0))}{r^{2s}}
   .
$$
As $r$ can be chosen arbitrarily big  this proves that $P^2_{n,s}(\Omega)=0.$ By definition $P^1_{n,s}(\Omega)\leq P^2_{n,s}(\Omega),$ so we also have that $P^1_{n,s}(\Omega)=0.$ 
\end{proof}

The following lemma relates the constants $C_{n,s}$ for different values of the dimension $n.$ 
It is a generalization of the Lemma 3.1 in \cite{pi}.
We will use these algebraic relations several times.

\begin{lemma}
\label{reduction formula}
    For each $m, n\in\N$ with $1\leq m<n$  and $0<s<1$, let $C_{n,s}$ be the constant \eqref{const_flap} appearing in the definition of the $[\cdot]_{s,2,\Omega}$ norm. The following two identities hold:
\smallskip
    
(i) $
    C_{n,s}\Theta_{m,n}=C_{n-m,s},
    \quad\text{ where }\quad \Theta_{m,n}=\mathcal{H}^{m-1}
    \left(\sph^{m-1}\right)\int_0^\infty\frac{t^{m-1}}   
    {(1+t^2)^\frac{n+2s}{2}}\;dt
$
where $\sph^{m-1}$ is the unit sphere in the Euclidean space $\re^m.$
\smallskip

(ii) If $a>0$ and $z\in\R^m$ then
\begin{align*}
  \int_{\R^m}\frac{dx}{\left(1+\frac{|x-z|^2}{a^2}\right)^{\frac{n+2s}{s}}}
  =a^m \Theta_{m,n}.
\end{align*}

\end{lemma}

\begin{proof}
Using the  change of variables $t=\tan\theta$ in the expression $\Theta_{m,n}$, we obtain 
\begin{align*}
     \Theta_{m,n} =   
     \int_{\sph^{m-1}} d\sigma 
     \int_0^\frac{\pi}{2}(\sin\theta)^{m-1}(\cos\theta)^{n-m+2s-1}\;d\theta
      =\frac{1}{2}B\bigg(\frac{m}{2},\frac{n-m+2s}{2}\bigg) 
      \frac{2\pi^{\frac{m}{2}}}{\Gamma(\frac{m}{2})}.
\end{align*}
From the definition of $C_{n,s}$ in \eqref{const_flap}, and formulas \eqref{eq:surface of S n-1} and \eqref{eq:properties of Beta} we get the desired result.
\smallskip 

(ii) The integral clearly does not depend on $z.$ So taking $z=0,$ the identity follows  immediately by change of variables and radial integration.
\end{proof}

\smallskip

In what follows we will use the following abbreviations, for $x=(x_1,\ldots,x_n)\in\R^n$ we write
$ x=(x_1,x')$. We next prove an important lemma regarding domain symmetrization. In addition to its usefulness in proving Theorem \ref{theorem:general s smaller half in tube}, this may have an independent interest when dealing with regional fractional Poincar\'e and fractional Poincar\'e inequalities for general domains. 
 
\begin{definition}
\label{def:cylindrical Schwarz symm}
Let $\Omega\subset\R^n$ be a measurable set. We define $\Omega^*,$ its \textit{cylindrical Schwarz symmetrization}, as the set which is rotationally symmetric with respect to the $x_1$ axis and $\mathcal{H}^{n-1}(\Omega\cap \{x_1=R\})=\mathcal{H}^{n-1}(\Omega^*\cap \{x_1=R\})$ for all $R\in\R,$ More precisely,
$$
    \Omega^*=
    \left\{(x_1,x')\in\R^n:\, x'\in \left(\Omega_{x_1}\right)_{n-1}^{\star}\right\}, \text{ where }
    \Omega_{x_1}
    =
    \{x'\in\R^{n-1}:\,(x_1,x')\in\Omega\}
$$
and $(A)_{n-1}^{\star}$ is the standard Schwarz symmetrization of $A$ in $\R^{n-1},$ i.e. $A$ is replaced by a ball of same $\mathcal{L}^{n-1}$ measure and centered at the origin.
\end{definition}

\begin{lemma}
\label{lemma:cylindrical Schwarz symm}
Let $\Omega\subset\R^n$ be a measurable set and $0< s.$ Then for any two disjoint sets $I,J\subset\R$ 
\begin{equation}
 \label{lemma:eq:schwarz symm}
   \int_{\Omega\cap\{x_1\in I\}}dx\int_{\Omega\cap\{x_1\in J\}}dy
   \frac{1}{|x-y|^{n+2s}}
   \leq
   \int_{\Omega^*\cap\{x_1\in I\}}dx\int_{\Omega^*\cap\{x_1\in J\}}dy
   \frac{1}{|x-y|^{n+2s}}.
\end{equation}
\end{lemma}

\begin{proof}
We write the left hand side of \eqref{lemma:eq:schwarz symm} as
\begin{align*}
   \int_{I}dx_1\int_{\Omega_{x_1}}dx'
   \int_Jdy_1\int_{\Omega_{y_1}}dy'\frac{1}{|x-y|^{n+2s}}
   =&
    \int_Idx_1
   \int_Jdy_1\int_{\Omega_{x_1}}dx'\int_{\Omega_{y_1}}
   dy'\frac{1}{|x-y|^{n+2s}}.
\end{align*}
Let $\chi_A$ denote the characteristic function of  a set $A$ and abbreviate for each fixed $(x_1,y_1)$ the function
$$
    h_{(x_1,y_1)}(z):= \frac{1}{\left((x_1-y_1)^2+|z|^2\right)^{(n+2s)/2}},
    \qquad z\in\R^{n-1}.
$$
 $h$ is a radially decreasing symmetric function of $z$, so that $h^{\star}=h.$
Then the left hand side of \eqref{lemma:eq:schwarz symm}
can be written as
$$
   \int_I dx_1
   \int_J dy_1\left(\int_{\R^{n-1}}\int_{\R^{n-1}}\chi_{\Omega_{x_1}}(x')
   \chi_{\Omega_{y_1}}(y')
   h_{(x_1,y_1)}(|x'-y'|)dx'dy'\right).
$$
In the same way the right hand side of \eqref{lemma:eq:schwarz symm} can be expressed, by replacing $\Omega$ by $\Omega^*.$ Thus the lemma follows from the Riesz rearrangement inequality.
\end{proof}
\section{Proof of Theorem \ref{Poincare 1} and Theorem \ref{Poincare gen domain 1}}
\label{section:Proof 1.1 and 1.2}

We start by giving some propositions and lemmas that will be useful in the proof of Theorem \ref{Poincare 1} and Theorem \ref{Poincare gen domain 1}.

\begin{proposition}\label{useful prop}
   Let  $1\leq m <n $ be two integers, $0<s<1$ and $\Omega_{\infty}=\R^m\times\omega$ be a strip in $\R^n$ where $\omega\subset\R^{n-m}$ is a bounded open set. Then we have 
    \begin{align*}
    P^1_{n,s}(\omegainfty)\leq P^1_{n-m,s}(\omega)\, .
    \end{align*}
    
        
\end{proposition}

\begin{proof}
\textit{Step 1:} It is sufficient to show that for any $W\in  C_c^{\infty}(\omega)$ and $\epsilon>0$ there exists $u\in C_c^{\infty}(\Omega_{\infty})$ such that
$$
   \frac{[u]_{s,2,\Omega_{\infty}}^2}{\|u\|^2_{L^2(\Omega_{\infty})}}\leq
   \frac{[W]^2_{s,2,\omega}}{\|W\|^2_{L^2(\omega)}}+\epsilon.
$$
Take a function $v\in C_c^{\infty}(\R^m)$ such that
$$
    \int_{\R^m}|v|^2=1\quad\text{ and define for $\ell>0$ }\quad v_\ell(x)
    =\ell^{\frac{-m}{2}}v\left(\frac{x}{\ell}\right).
$$
Then $v_\ell\in C_c^{\infty}(\R^m)$ and 
\begin{equation}
 \label{eq:vl square integral is one}
   \int_{\R^m}|v_{\ell}|^2=1\quad\text{ for all }\ell.
\end{equation}
We shall use the notation $x=(X_1,X_2)\in \R^n,$ $X_1\in \R^m$ and $X_2\in \R^{n-m}.$ At last we define
$$
    u_{\ell}(X_1,X_2)=v_{\ell}(X_1)W(X_2),
$$
and claim that for $\ell$ big enough $u_{\ell}$ has the desired property. Without loss of generality we can assume that $\|W\|_{L^2(\omega)}=1.$ Therefore, using \eqref{eq:vl square integral is one} we get that
$$
    \|u_{\ell}\|_{L^2(\Omega_{\infty})}^2=
    \int_{\R^m}\int_{\omega}v_{\ell}^2(X_1)W^2(X_2)dX_2\,dX_1
    =\|W\|_{L^2(\omega)}=1.
$$
We therefore have to prove that, for some $\ell$ big enough
\begin{equation*}
 \label{eq:norm infty leq omeganorm plus eps}
     [u_{\ell}]^2_{s,2,\Omega_{\infty}}\leq
   [W]^2_{s,2,\omega}+\epsilon.
\end{equation*}
Using that
\begin{align*}
  |u_{\ell}(x)-u_{\ell}(y)|^2
  =&
  \left|v_{\ell}(X_1)W(X_2)-v_{\ell}(Y_1)W(X_2)+
  v_{\ell}(Y_1)W(X_2)-v_{\ell}(Y_1)W(Y_2)\right|^2
  \smallskip 
  \\
  =&v_{\ell}^2(Y_1)\left|W(X_2)-W(Y_2)\right|^2+\left|v_{\ell}(X_1)-v_{\ell}(Y_1)\right|^2
  W^2(X_2) 
  \smallskip
  \\
  & \hspace{1cm} + 2 v_{\ell}(Y_1)W(X_2)\left(W(X_2)-W(Y_2)\right)\,
  \left(v_{\ell}(X_1)-v_{\ell}(Y_1)\right),
\end{align*}
we write $[u_{\ell}]^2_{s,2,\Omega_{\infty}}$ as
$
   [u_{\ell}]_{s,2,\Omega_{\infty}}^2 :=I_1+I_2+I_3\,,
$
where
$$
   I_1=\frac{C_{n,s}}{2}
   \int_{\omegainfty}\int_{\omegainfty}\frac{|{v_\ell}(Y_1)(W(X_2)-W(Y_2))|^2}
   {|x-y|^{n+2s}}
   dx\;dy,
$$
and 
\begin{align*}
  I_2=& \frac{C_{n,s}}{2}
   \int_{\omegainfty}\int_{\omegainfty}\frac{\left|v_{\ell}(X_1)-v_{\ell}(Y_1)\right|^2
  W^2(X_2) }
   {|x-y|^{n+2s}}
   dx\;dy ,
   \smallskip \\
   I_3=& 
   C_{n,s}
   \int_{\omegainfty}\int_{\omegainfty}\frac{ v_{\ell}
   (Y_1)W(X_2)\left(W(X_2)-W(Y_2)\right)\,
  \left(v_{\ell}(X_1)-v_{\ell}(Y_1)\right)}
   {|x-y|^{n+2s}}
   dx\;dy.
\end{align*}
We will show that
\begin{equation}
 \label{I1 I2 and I3}
    I_1=[W]_{s,2,\omega}^2\quad\text{ and }\quad I_2\,,I_3\to 0
    \quad\text{ as }\ell\to \infty.
\end{equation}

\textit{Step 2 (Calculation of $I_1$):} We obtain from  Lemma \ref{reduction formula} (ii), whenever $X_2\neq Y_2$,
\begin{align*}
  \int_{\R^m}\frac{dX_1}{\left(1+\frac{|X_1-Y_1|^2}{|X_2-Y_2|^2}\right)^{\frac{n+2s}{2}}}
    =|X_2-Y_2|^m \Theta_{m,n}\hspace{3mm} \textrm{for any} \  Y_1\in \R^m.
\end{align*}
Plugging this identity into the definition of $I_1$, using Lemma \ref{reduction formula} (i) and then \eqref{eq:vl square integral is one} gives
\begin{align*}
  I_1=&
  \frac{C_{n,s}}{2}
  \int_{\omegainfty}\int_{\omegainfty}\frac{|{v_\ell}
  (Y_1)(W(X_2)-W(Y_2))|^2}{|X_2-Y_2|^{n+2s}
  \big(1+\frac{|X_1-Y_1|^2}{|X_2-Y_2|^2}\big)^{\frac{n+2s}{2}}}dxdy 
  \smallskip
  \\
    =&\frac{C_{n,s}}{2} 
	\int_{\omega}\int_{\omega}\frac{|W(X_2)-W(Y_2)|^2}{|X_2-Y_2|^{n+2s}}\int_{\re^m}   
	 \left(\int_{\re^m}\frac{dX_1}{\big(1+\frac{|X_1-Y_1|^2}{| 
	X_2-Y_2|^2}\big)^{\frac{n+2s}{2}}}\right)|v_\ell(Y_1)|^2dY_1dX_2dY_2 
	\smallskip
	\\
	=&
	\frac{C_{n,s}}{2}\Theta_{m,n}\int_{\omega}\int_{\omega}
	\frac{\left|W(X_2)-W(Y_2)\right|^2}{\left|X_2-Y_2\right|^{n-m+2s}}dX_2dY_2
	\int_{\R^m}|v_{\ell}(Y_1)|^2dY_1
	\smallskip
	\\
	=&
	\frac{C_{n-m,s}}{2}\int_{\omega}\int_{\omega}
	\frac{\left|W(X_2)-W(Y_2)\right|^2}{\left|X_2-Y_2\right|^{n-m+2s}}dX_2dY_2
	=[W]_{s,2,\omega}^2\,.
\end{align*}
This proves the first statement of \eqref{I1 I2 and I3}.
\smallskip

\textit{Step 3 (Estimates for $I_2$ and $I_3$):} We write $I_2$ as
\begin{align*}
   I_2&=\frac{C_{n,s}}{2}\int_{\omegainfty}\int_{\omegainfty}\frac{|({v_\ell}(X_1)-{v_\ell}(Y_1))W(X_2)|^2}{|X_1-Y_1|^{n+2s}\bigg(1+\frac{|X_2-Y_2|^2}{|X_1-Y_1|^2}\bigg)^{\frac{n+2s}{2}}}dxdy 
	\smallskip 
	\\
		&=\frac{C_{n,s}}{2}
		\int_{\re^m}\int_{\re^m}\frac{|{v_\ell}(X_1)-{v_\ell}(Y_1)|^2}
		{|X_1-Y_1|^{n+2s}}\int_{\omega} 
	\Bigg(\int_{\omega}\frac{dY_2}
	{\big(1+\frac{|X_2-Y_2|^2}{|X_1-Y_1|^2}\big)^{\frac{n+2s}{2}}}\Bigg)
	|W(X_2)|^2\;dX_2dX_1dY_1\,.
\end{align*}
Using Lemma \ref{reduction formula} (ii) we get
\begin{align*}
   \int_{\omega}\frac{dY_2}
	{\big(1+\frac{|X_2-Y_2|^2}{|X_1-Y_1|^2}\big)^{\frac{n+2s}{2}}}
	\leq &
	\int_{\R^{n-m}}\frac{dY_2}
	{\big(1+\frac{|X_2-Y_2|^2}{|X_1-Y_1|^2}\big)^{\frac{n+2s}{2}}}
	=|X_1-Y_1|^{n-m}\Theta_{n-m,n}.
\end{align*}
Plugging this into the definition of $I_2$ and using once more that $\|W\|_{L^2}=1$ gives
$$
    I_2 \leq \frac{C_{n,s}\;\Theta_{n-m,n}}{2}\int_{\R^m}\int_{\R^m}\frac{\left|v_{\ell}(X_1)-v_{\ell}(Y_1)\right|^2}
    {|X_1-Y_1|^{m+2s}} dX_1dY_1=[v_{\ell}]_{s,2,\R^m}^2.
$$
By Proposition \ref{prop:elementary properties} (ii) and definition of $v_{\ell}$
$$
    [v_{\ell}]_{s,2,\R^m}^2=\frac{\ell^{m-2s}}{\ell^m}[v]_{s,2,\R^m}^2=
    \frac{1}{\ell^{2s}}[v]_{s,2,\R^m}^2\quad\Rightarrow\quad I_2
    \leq \frac{[v]_{s,2,\R^m}^2}{\ell^{2s}},
$$
which proves \eqref{I1 I2 and I3}  for $I_2$.
For $I_3$ we use H\"older inequality and estimate it as
\begin{align*}
   I_3\leq &
   2 \sqrt{I_1}\sqrt{I_2}\leq 2 [W]_{s,2,\omega} \frac{[v]_{s,2,\R^m}}{\ell^s}.
\end{align*}
This shows \eqref{I1 I2 and I3}  also for $I_3$.
\end{proof}
 \smallskip

For the proofs of Theorem \ref {Poincare 1} and \ref {Poincare gen domain 1}
we will use the following lemma, which follows from an appropriate integration in spherical coordinates and an application of the change of variables formula.

\begin{lemma}[Loss and Sloane \cite{loss}, Lemma 2.4]
\label{lemma:Loss Sloan}
Let $p>0,$ $0<s<1$ and $\Omega\subset\R^n$ be a measurable set. Then for any $u\in C_c^{\infty}(\Omega)$
\begin{align*}
     &2\int_{\Omega}dx\int_{\Omega}dy
     \frac{|u(x)-u(y)|^p}{|x-y|^{n+sp}}
     \smallskip
     \\
    =&\int_{\sph^{n-1}}d\mathcal{H}^{n-1}(w)
    \int_{\{x:\;x\cdot w=0\}}d\mathcal{H}^{n-1}(x)
    \int_{\{\ell:\,x+\ell w\in\Omega\}}
    \int_{\{t:\,x+tw\in\Omega\}}\frac{|u(x+\ell w)-u(x+tw)|^p}{|\ell-t|^{1+sp}}dtd\ell.
\end{align*}
\end{lemma}

\begin{lemma}
\label{lemma:angle condition}
Let $\frac{1}{2}<s<1$ and $\Omega\subset\R^n$ be a measurable set. Suppose that there exists a $\mathcal{H}^{n-1}$ measurable function $f:\mathbb{S}^{n-1}\to [0,\infty)$ such that 
$$
    P^1_{1,s}\left(\Omega\cap \{x+tw:\,t\in\R\}\right)\geq f(w)
$$
for a.e. $w\in\mathbb{S}^{n-1}$ and a.e. $x\in \{y\in \R^n:\,y\cdot w=0\}.$ Then it holds that
$$
    P^1_{n,s}(\Omega)\geq \frac{C_{n,s}}{2C_{1,s}}\int_{\mathbb{S}^{n-1}}f(w)d\mathcal{H}^{n-1}(w).
$$
\end{lemma}

\begin{proof}
Let $w\in\mathbb{S}^{n-1}$ and $x\in L_w:= \{y\in \R^n:\,y\cdot w=0\}.$ Define
$
   \Omega_{w,x}:=\Omega\cap \{x+t w:\,t\in\R\}.
$
Then by  hypothesis
\begin{align*}
    \frac{C_{1,s}}{2}\int_{\{\ell:\,x+\ell w\in\Omega\}}
    \int_{\{t:\,x+tw\in\Omega\}}\frac{|u(x+\ell w)-u(x+tw)|^2}{|\ell-t|^{1+2s}}dtd\ell
    \geq 
    &P^1_{1,s}(\Omega_{w,x})\int_{\Omega_{w,x}}|u(x+tw)|^2dt
    \smallskip
    \\
    \geq & f(w) \int_{\Omega_{w,x}}|u(x+tw)|^2dt.
\end{align*}
Note that for any $w\in \mathbb{S}^{n-1}$, by Fubini (or change of variables)
$$
    \int_{L_w}d\mathcal{H}^{n-1}(x)\int_{\Omega_{x,w}}|u(x+tw)|^2dt=
    \int_{\Omega}|u|^2.
$$
Therefore it follows from Lemma \ref {lemma:Loss Sloan} that
$$
    [u]_{W^{s,2}(\Omega)}^2\geq \frac{C_{n,s}}{2C_{1,s}}
   \left( \int_{\mathbb{S}^{n-1}}f(w)d\mathcal{H}^{n-1}(w)\right)\int_{\Omega}|u|^2,
$$
which proves the lemma.
\end{proof}

We will use the explicit form of the following hyperspherical coordinates and their properties. Let us define $Q_{n-1}\subset\re^{n-1}$ by
$$
  Q_{n-1}=(0,\pi)^{n-2}\times (0,2\pi).
$$
The hyper spherical coordinates  $H=(H_1,\ldots,H_n):Q_{n-1}\to \mathbb{S}^{n-1}$ are defined as: for $k=1,\ldots,n$ and $\varphi=(\varphi_1,\ldots,\varphi_{n-1})$
$$
  H_k(\varphi)=\cos\varphi_k\prod_{l=0}^{k-1}\sin\varphi_l\quad\text{ with convention }\varphi_0=\frac{\pi}{2},\quad\varphi_n=0.
$$
A calculation shows that
$
  d_i(\varphi):=\left\langle \frac{\partial H}{\partial \varphi_i},\frac{\partial H}{\partial \varphi_i}\right\rangle =\prod_{l=0}^{i-1}\sin^2\varphi_l>0.
$
One verifies that the metric tensor  in these coordinates is diagonal
$
  g_{ij}(\varphi)=\left\langle\frac{\partial H}{\partial \varphi_i},\frac{\partial H}{\partial \varphi_j}\right\rangle=\delta_{ij}d_i(\varphi),
$
($\delta_{ij}=1$ if $i=j$ and $0$ else)
and hence the surface element $g_{n-1}$ is given by
$$
  g_{n-1}(\varphi)=\sqrt{\det g_{ij}(\varphi)}=\prod_{k=1}^{n-1}d_k(\varphi)= \prod_{k=1}^{n-2}(\sin\varphi_k)^{n-k-1}.
$$
Note that for any function $f$ depending only of $\varphi_1$ we have that
\begin{align*}
    \int_{Q_{n-1}}f(\varphi_1)g_{n-1}(\varphi)d\varphi=
    &\int_0^{\pi}
    f(\varphi_1)(\sin\varphi_1)^{n-2}\left(\int_{Q_{n-2}} 
    (\sin\varphi_2)^{n-3}\cdots\sin\varphi_{n-2}d\varphi_2\cdots d\varphi_{n-2}\right)
    d\varphi_1
    \smallskip
    \\
    =&\int_0^{\pi} f(\varphi_1)(\sin\varphi_1)^{n-2}\left(\int_{Q_{n-2}} 
   g_{n-2}(\theta)d\theta\right)
    d\varphi_1
    \smallskip
    \\
    =&\mathcal{H}^{n-2}(\mathbb{S}^{n-2})
    \int_0^{\pi} f(\varphi_1)(\sin\varphi_1)^{n-2}d\varphi_1.
\end{align*}
In particular for $f(\varphi)=|\cos\varphi_1|^{2s}$ we obtain, using \eqref{eq:properties of Beta},  that
\begin{equation}
 \label{eq:cos varphi 1 integral}
  \begin{split}
       \int_{Q_{n-1}}|\cos\varphi_1|^{2s}g_{n-1}(\varphi)d\varphi=&
       2\mathcal{H}^{n-2}(\mathbb{S}^{n-2})
    \int_0^{\frac{\pi}{2}} (\cos\varphi_1)^{2s}(\sin\varphi_1)^{n-2}d\varphi_1
    \smallskip
    \\
    =&\frac{2\pi^{\frac{n-1}{2}}}{\Gamma\left(\frac{n-1}{2}\right)}
    B\left(\frac{n-1}{2},\frac{2s+1}{2}\right).  
   \end{split}  
\end{equation}

\begin{proof}[\textbf{Proof of Theorem \ref{Poincare 1}}]
\textit{Part (1):}
We apply the above Proposition \ref{useful prop} by choosing $m=n-1$ and $\omega=(-1,1) \subset \re$. In particular we have 
$
P^1_{n,s}(\omegainfty)\leq P^1_{1,s}((-1,1)).$
 Now use that $P^1_{1,s}((-1,1))=0$ for any $s\in (0, \frac{1}{2}]$ [see, Proposition \ref{prop:elementary properties} (i)]. 
\smallskip
    
\textit{Part (2):} By Proposition \ref{useful prop} we know that $P_{n,s}^1(\Omega_{\infty})\leq P_{1,s}^1((-1,1)).$ So it is sufficient to show that
\begin{equation}
  \label{eq:P omegainfty geq P crosssec}
    P_{n,s}^1(\Omega_{\infty})\geq P_{1,s}^1((-1,1)).
\end{equation}
We will deduce this inequality from Lemma \eqref{lemma:angle condition}. Let $w=(w_1,\ldots,w_n)\in \mathbb{S}^{n-1}$ be such $w_1\neq 0.$ By the special form of $\Omega_{\infty}=(-1,1)\times\R^{n-1}$ we have that  the length of the intersection $\Omega_{\infty}\cap \{x+tw:\,t\in\R\}$ does not depend on $x.$ So we obtain that
\begin{align*}
   \mathcal{H}^1\left(\Omega_{\infty}\cap \{x+tw:\,t\in\R\}\right)
   =&
     \mathcal{H}^1\left(\Omega_{\infty}\cap \{(-1,0,\ldots,0)+tw:\,t\in\R\}\right)
     \smallskip
     \\
     =&| t_0(w)|\quad\text{ where } -1+t_0(w)w_1=1\quad\Rightarrow\quad 
     t_0(w)=\frac{2}{w_1}.
\end{align*}
From Proposition \ref{prop:elementary properties} (ii) we obtain that
$$
    P_{1,s}^1\left(\Omega_{\infty}\cap \{x+tw:\,t\in\R\}\right)
    =\left(\frac{|w_1|}{2}\right)^{2s}P_{1,s}^1((0,1))
    =|w_1|^{2s}P_{1,s}^1((-1,1)).
$$
Using Lemma \ref{lemma:angle condition}, the hyperspherical coordinates and \eqref{eq:cos varphi 1 integral} gives
\begin{align*}
   P_{n,s}^1(\omegainfty)
   \geq &
   P_{1,s}^1((-1,1))\frac{C_{n,s}}{2 C_{1,s}}\int_{\mathbb{S}^{n-1}}
   |w_1|^{2s}d\mathcal{H}^{n-1}\\
   =&
    P_{1,s}^1((-1,1))\frac{C_{n,s}}{2 C_{1,s}}\int_{Q_{n-1}}
   |\cos\varphi_1|^{2s}g_{n-1}(\varphi)d\varphi
   \smallskip
   \\
   =&   P_{1,s}^1((-1,1))\frac{C_{n,s}}{2 C_{1,s}}
   \frac{2\pi^{\frac{n-1}{2}}}{\Gamma\left(\frac{n-1}{2}\right)}
    B\left(\frac{n-1}{2},\frac{2s+1}{2}\right).
\end{align*}
It is immediate to verify, using \eqref{eq:properties of Beta}, that
$$
   \frac{C_{n,s}}{ C_{1,s}}
   \frac{\pi^{\frac{n-1}{2}}}{\Gamma\left(\frac{n-1}{2}\right)}
    B\left(\frac{n-1}{2},\frac{2s+1}{2}\right)=1,
$$
which concludes the proof of  \eqref{eq:P omegainfty geq P crosssec}.  
\end{proof}

\smallskip

\begin{proof}[\textbf{Proof of Theorem \ref{Poincare gen domain 1}}]
By hypothesis, Remark \ref {remark:finite ball cond in 1 dim}, and \eqref{eq:P11 estimated by m} we obtain that
$$
    P^1_{1,s}(\Omega\cap \{x+tw:\,t\in\R\})\geq \frac{P_{1,s}^1((0,1))}{m^{2s}}
    \quad\text{ for all }w\in \Sigma,\; x\in\R^n.
$$
Thus if we define
$$
    f(w)=\left\{\begin{array}{rl}
             \frac{P_{1,s}^1((0,1))}{m^{2s}} & \text{ if } w\in\Sigma
             \smallskip \\
             0 & \text{ if }w\in \mathbb{S}^{n-1}\setminus\Sigma
          \end{array}\right.,
$$
then $f$ satisfies the hypothesis of Lemma \ref{lemma:angle condition} and we get that
$$
   P_{n,s}^1(\Omega)\geq \frac{C_{n,s}}{2C_{1,s}}\mathcal{H}^{n-1}
   (\Sigma)\frac{P_{1,s}^1((0,1))}{m^{2s}}.
$$
This completes the proof of the theorem.\end{proof}

We will provide some examples of the domains which satisfy the hypothesis of Theorem \ref{Poincare gen domain 1} and one which does not.

\begin{example}
\label{exmpl:domains with sufficient condn}
(i) \textit{Domain between Graphs:}
Let $f_1,f_2:\mathbb R^{n-1}\to[m,M]$ be two bounded continuous function such that $f_1 < f_2$. $\Omega $ is defined as  $$\Omega=\{(x,y)\in\mathbb R^n:f_1(x)<y< f_2(x)\}.$$

\smallskip

(ii) \textit{Finite union of strips:} The domain $\Omega$ is the finite union of the strips in $\rn.$ Let $O(n)$ denote the set on $n\times n$ orthogonal(rotation) matrices.   Given $A_i\in O(n)$, $z_i \in \re^n $, $b_i\in \R$ for $i=1,\ldots,M$ and $\omegainfty=(-1,1)\times \R^{n-1}$, then define
$$
   \Omega=\bigcup_{i=1}^M \left( b_i A_i(\omegainfty)+ z_i \right).
$$

\smallskip

(iii) \textit{Infinitely many parallel strips:}
$\Omega=\left(\bigcup_{i=1}^{\infty}I_i\right)\times\R^{n-1}\subset\R^n,$ where $I_i$ are disjoint open intervals and length $(I_i)\leq m$ for some uniform upper bound $m.$
\smallskip

(iv)  \textit{Infinite ``L" type domain:} $$\Omega =  \left( (0,1)\times (0,\infty)\right) \bigcup \left( (0,\infty) \times (0,1) \right)\subset\mathbb{R}^2.$$

(v) \textit{Concentric Annulus:} The following domain satisfies the finite ball condition but does not satisfy the hypothesis of Theorem \ref{Poincare gen domain 1}:
$$
   \Omega=\bigcup_{k=1}^{\infty}B_{2k}(0)\setminus B_{2k-1}(0).
$$
\end{example}

\section{Proof of Theorem \ref{best constant}}

The main tool  to prove  Theorem \ref{best constant} is to use discrete version of fractional Picone identity.  There are various version of Picone identity, we refer to [\cite{crs}, equation no 6.12]  for the version which is particularly helpful  for problems  involving second order elliptic operator. For study of fractional eigenvalue problem we refer to \cite{raffela}.

\begin{lemma}[Discrete Picone inequality]
Let $u,\;v$ be two measurable functions with $u>0$ and $v\geq 0$. Then
$$\big(u(x)-u(y)\big)\bigg[\frac{v^2(x)}{u(x)}-\frac{v^2(y)}{u(y)}\bigg]\leq|v(x)-v(y)|^2.$$
\end{lemma}

Expanding the left side of this inequality and by using Young's inequality we get the desired result. For further generalization,  we refer to the Proposition 4.2 in \cite{picone}. We will use the following abbreviation
$  
    x=(X_1,X_2)\in\R^n,  X_1\in \R^m,  X_2\in\R^{n-m}.
$
$\Omega_{\infty}=\R^m\times\omega,$ $\omega\subset\R^{n-m}$ shall always denote the sets defined in Theorem \ref{best constant}. We assume that $W$ is the first Eigenfunction of the fractional Laplace operator in $\omega.$ Thus (see \cite{strict positive})  $W$ is continuous in $\overline{\omega},$ smooth in the interior of $\omega,$  strictly positive in $\omega$ and satisfies for some $P_{n-m,s}^2(\omega)>0$
$$
    \left\{
    \begin{array}{rl}
    (-\Delta_{n-m})^sW=P_{n-m,s}^2(\omega)W& \quad\text{ in }\omega,
    \\
    W=0& \quad \text{ in }\R^{n-m}\setminus \omega,
    \\
    W>0& \quad\text{ in }\omega.
    \end{array}\right.
$$

\begin{lemma}\label{entire strip statisfy}
Let $x=(X_1,X_2)\in\omegainfty$ and define $u^*(x):=W(X_2).$ Then $(-\Delta_n)^su^*=P^2_{n-m,s}(\omega)\;u^*$ in $\omegainfty.$
\end{lemma}

\begin{proof}
Using first Lemma \ref{reduction formula} (ii) and then (i) we get 
\begin{align*}
  (-\Delta_n)^su^*(x)
  =&C_{n,s}\int_{\rn}\frac{u^*(x)-u^*(y)}{|x-y|^{n+2s}}dy 
  =C_{n,s}\int_{\rn}\frac{W(X_2)-W(Y_2)}{|x-y|^{n+2s}}dy\\
  =&C_{n,s}\int_{\re^{n-m}}dY_2\frac{W(X_2)-W(Y_2)}{|X_2-Y_2|^{n+2s}}
  \int_{\re^m}\frac{dY_1}{\bigg(1+\frac{|X_1-Y_1|^2}{|X_2-Y_2|^2}\bigg)^{\frac{n+2s}{2}}}
   \smallskip
   \\
   =&
   C_{n,s}\Theta_{m,n}\int_{\R^{n-m}}\frac{W(X_2)-W(Y_2)}{|X_2-Y_2|^{n-m+2s}}
   dY_2
   =\left(-\Delta_{n-m}\right)^sW(X_2)
   \smallskip
   \\
   =& P^2_{n-m,s}(\omega)W(X_2)
   =P^2_{n-m,s}(\omega)\;u^*(x).
  \end{align*}
\end{proof}

\begin{proof}[\textbf{Proof of the Theorem \ref{best constant}}.]

Since $u^*(x):=W(X_2) \in C^\infty(\omegainfty)$ is strictly positive in $\omegainfty$. We have  for any $v\in C_c^\infty(\omegainfty)$ the function $\phi=\frac{v^2}{u^*}$ belongs to  $\phi\in C_c^\infty(\omegainfty)$. 
Then by Discrete Picone inequality we have 
$$
    \big((u^*(x)-u^*(y)\big)\big(\phi(x)-\phi(y)\big)=\big(u^*(x)-u^*(y)\big)
    \bigg[\frac{v^2(x)}{u^*(x)}-\frac{v^2(y)}{u^*(y)}\bigg]\leq|v(x)-v(y)|^2.
$$
Integrating two times over $\rn$ we obtain
$$
   \frac{C_{n,s}}{2}\;\int_{\rn}\int_{\rn}\frac{\big((u^*(x)-u^*(y)\big)
   \big(\phi(x)-\phi(y)\big)}{|x-y|^{n+2s}}\;dxdy\leq\frac{C_{n,s}}{2}
   \int_{\rn}\int_{\rn}\frac{|v(x)-v(y)|^2}{|x-y|^{n+2s}}\;dxdy.
$$
Writing the left hand side as sum of two integral, one containing $\phi(x)$ and the other $\phi(y)$ and making a change of variables in the second $x\mapsto y$ gives
$$
    C_{n,s}\int_{\rn}\int_{\rn}\frac{\big((u^*(x)-u^*(y)\big)\;\phi(x)}{|x-y|^{n+2s}}\;dxdy
    \leq  \frac{C_{n,s}}{2}\int_{\rn}\int_{\rn}\frac{|v(x)-v(y)|^2}{|x-y|^{n+2s}}\;dxdy.
$$
It follows that
$$
   C_{n,s}\int_{\omegainfty}\frac{v^2(x)}{u^*(x)}\int_{\rn}\frac{u^*(x)-u^*(y)}{|x-y|^{n+2s}}\;dydx\leq\frac{C_{n,s}}{2}\int_{\rn}\int_{\rn}\frac{|v(x)-v(y)|^2}{|x-y|^{n+2s}}\;dxdy.
$$
As this is true for any $v\in C_c^\infty(\omegainfty)$ and using Lemma \ref{entire strip statisfy}, we get that
$
    P^2_{n-m,s}(\omega)\leq P^2_{n,s}(\omegainfty).
$
For the reverse estimate one can  proceed exactly in the same way as in Proposition \ref{useful prop}.
\end{proof} 
\section{Sufficient conditions for $s\in (0,\frac{1}{2})$} \label{sec:sufficent_cndn_unbd_dom_0}

\subsection{Proof of Theorem \ref{theorem:distance condition} and related discussions}

In this section we  will denote
$\quad x'\in \R^{n-1}, \quad \{x_1=R\}=\{(x_1,x')\in\R^n:\, x_1=R\},
$
and similar notations for $x_1>R$. Theorem \ref{theorem:distance condition} is having a general and relatively abstract condition on domains. We discuss various examples satisfying the condition later. First, we prove this main theorem.

\begin{proof}[\textbf{Proof of Theorem \ref{theorem:distance condition}}]
Define $\Omega_k=\Omega\cap \lambda_k U$. Let $\epsilon>0$ be given. By hypothesis there exists $k\in\mathbb{N}$ such that
\begin{equation}
   \label{eq:hypothesis Omega k}
     \frac{1}{\mathcal{L}^n(\Omega_k)}\int_{\Omega_k}\frac{dx}{\dist(x,
     (\lambda_kU)^c)^{2s}}\leq \epsilon.
\end{equation}
By hypothesis $\Omega_k$ is bounded. $\Omega_k$ might not be Lipschitz, but $\mathcal{H}^{n-1}(\partial\Omega_k)\leq \mathcal{H}^{n-1}(\partial\Omega \, \cap \, \lambda_k \bar{U})+C\lambda_k^{n-1} <\infty.$ Therefore by Proposition \ref{prop:elementary properties} (i) and Remark \ref {remark:bounded perimeter} there exists a $v_k\in C_c^{\infty}(\Omega_k)$ such that
\begin{equation}
 \label{eq:vk in bounded Omega k, for thm 1.3}
    \frac{1}{\|v_k\|_{L^2(\Omega_k)}^2}\int_{\Omega_k}\int_{\Omega_k}
    \frac{|v_k(x)-v_k(y)|^2}{|x-y|^{n+2s}}dx\,dy\leq \epsilon.
\end{equation}
Recall, by the proof of Proposition  \ref{prop:elementary properties} (i), that $v_k$ is an approximation of the characteristic function of $\Omega_k$ by cutting it off near the boundary. Thus we can also assume that
\begin{equation*}\label{eq:4.2.b}
    |v_k|\leq 1\quad\text{  and }
    \quad \|v_k\|_{L^2(\Omega_k)}^2\geq \frac{\mathcal{L}^n(\Omega_k)}{2}.
\end{equation*}
Finally define $u_k:=v_k$ in $\Omega_k$ and $u_k:=0$ in $\Omega\setminus\Omega_k$. Then $u_k\in C_c^{\infty}(\Omega)$ and
\begin{equation}
  \label{eq:4.2.c}
     \|u_k\|_{L^2(\Omega_k)}^2= \|v_k\|_{L^2(\Omega_k)}^2
     \geq  \frac{\mathcal{L}^n(\Omega_k)}{2} .
\end{equation}
Therefore, we get that 
$$
     \int_{\Omega_k}dx\int_{\Omega\setminus \Omega_k}dy
    \frac{|u_k(x)-u_k(y)|^2}{|x-y|^{n+2s}}
    \leq 
    \int_{\Omega_k}dx\int_{\Omega\setminus \Omega_k}dy
    \frac{1}{|x-y|^{n+2s}}.
$$
Now use that for $x\in\Omega_k$
\begin{align*}
   \int_{\Omega\setminus \Omega_k}
    \frac{dy}{|x-y|^{n+2s}}
    \leq &
    \int_{\{y:\,|x-y|>\dist(x,(\lambda_kU)^c)\}} \frac{dy}{|x-y|^{n+2s}}
    =c(n)\int_{\dist(x,(\lambda_kU)^c)}^{\infty}\frac{r^{n-1}}{r^{n+2s}}dr
    \smallskip
    \\
    =&\frac{c(n,s)}{\dist(x,\ (\lambda_kU)^c)^{2s}}.
\end{align*}
Plugging this into the previous inequality and using 
\eqref{eq:4.2.c},  \eqref{eq:hypothesis Omega k}  gives that
\begin{equation}\label{uk bound compl omegak}
 \frac{1}{\|u_k\|_{L^2(\Omega_k)}^2}
    \int_{\Omega_k}dx\int_{\Omega\setminus \Omega_k}dy
    \frac{|u_k(x)-u_k(y)|^2}{|x-y|^{n+2s}} \\
      \leq 
     \frac{2 c(n,s)}{\mathcal{L}^n(\Omega_k)}\int_{\Omega_k}\frac{dx}
     {\dist(x,(\lambda_kU)^c)^{2s}}
     \leq 2 c(n,s)\epsilon.
\end{equation}
Using \eqref{eq:vk in bounded Omega k, for thm 1.3} and \eqref{uk bound compl omegak} we  conclude that, 
\begin{align*}
&\frac{1}{\|u_k\|_{L^2(\Omega)}^2}\int_{\Omega}\int_{\Omega}
    \frac{|u_k(x)-u_k(y)|^2}{|x-y|^{n+2s}}dx\,dy
    \smallskip 
    \\
   & = 
    \frac{1}{\|v_k\|_{L^2(\Omega)}^2}
    \int_{\Omega_k}\int_{\Omega_k}
    \frac{|v_k(x)-v_k(y)|^2}{|x-y|^{n+2s}}dx\,dy 
    +\frac{2}{\|u_k\|_{L^2(\Omega)}^2}
    \int_{\Omega_k}\int_{\Omega\setminus \Omega_k}
    \frac{|u_k(x)-u_k(y)|^2}{|x-y|^{n+2s}}dx\,dy
    \smallskip 
    \\
   & 
    \leq
    \epsilon(1+4c(n,s)).
\end{align*}
This completes the proof of the theorem.
\end{proof}
\begin{example}\label{example:union of 2 strips R3}
Let $\omegainfty$ is the union of two perpendicular infinite strips in $\rn$, i.e. $$\omegainfty=\left(\re^{n-1}\times(-1,1)\right)\cup\left((-1,1)\times\re^{n-1}\right).$$
Then $P_{n,s}^1(\omegainfty) = 0$ whenever $s\in(0,\frac{1}{2})$.\smallskip

The idea is to apply  Theorem \ref{theorem:distance condition}  with  $\Omega= \omegainfty$ and $U=(-1,1)^n.$
Let $A:= \re^{n-1}\times(-1,1)$, $B:= (-1,1)\times\re^{n-1}$.  Then clearly, 

$$
    \int_{\lambda_k U\cap \Omega}
    \frac{dx}{\dist(x,(\lambda_kU)^c)^{2s}} \leq  \int_{\lambda_k U\cap A}
    \frac{dx}{\dist(x,(\lambda_kU)^c)^{2s}} +  \int_{\lambda_k U\cap B}
    \frac{dx}{\dist(x,(\lambda_kU)^c)^{2s}}.$$

 \noindent Using the symmetric property of the domain $\omegainfty$, it is sufficient to show that the following:
\begin{align}\label{example_plus_result}
    \lim_{k\to\infty}\frac{1}{\mathcal{L}^n(\lambda_k U\cap \Omega)}
    \int_{\lambda_k U\cap A}
    \frac{dx}{\dist(x,(\lambda_kU)^c)^{2s}}=0.
\end{align}

We verify this result for dimension $3$. Consider $\lambda_k=k \in \mathbb{N}$. Therefore,  $\lebmthree(\Omega\cap \lambda_k U) = Ck^2$ for some constant $C>0$ independent of $k$. Further, we write $\lambda_kU\cap A = A_1\cup A_2 \cup A_3$, where 
\begin{align*}
A_1=&  (A\cap \lambda_kU) \cap (-1,1)^3, \quad A_2=  \big((A\cap \lambda_kU)\setminus (-1,1)^3\big) \cap \{|x_1|>|x_2|\}, \\
A_3=&  \big((A\cap \lambda_kU)\setminus (-1,1)^3\big) \cap \{|x_2|>|x_1|\}.   
\end{align*} 
For, $x\in A_1$ we get that $\dist(x,(\lambda_kU)^c)\geq k-1$, then  
\begin{align}\label{example_plus_estimate1}
\frac{1}{\lebmthree(\lambda_k U\cap \Omega)} \int_{A_1} 
    \frac{dx}{\dist(x,(\lambda_kU)^c)^{2s}} \leq \frac{C}{k^2}\cdot \frac{\lebmthree(A_1)}{(k-1)^{2s}}\leq  \frac{C}{k^{2s+2}}.
\end{align}
 Again, 
\begin{align*}
\dist(x,(\lambda_kU)^c) = \left\{
\begin{array}{rl}
k-|x_1|, \quad &\mbox{if} \quad x\in A_2,\\
k-|x_2|, \quad &\mbox{if} \quad x\in A_3.
\end{array}\right.
\end{align*}
Then, 
\begin{align}\label{example_plus_estimate2}
&\notag \frac{1}{\lebmthree(\lambda_k U\cap \Omega)} \bigg( \int_{A_2} 
    \frac{dx}{\dist(x,(\lambda_kU)^c)^{2s}} + \int_{A_3} 
    \frac{dx}{\dist(x,(\lambda_kU)^c)^{2s}} \bigg) \\
\notag \leq & \frac{C}{k^2} \bigg( \int_{\{x_3<1\}} \int_{\{|x_2|<k\}}dx_3 dx_2\int_{\{|x_1|<k\}} \frac{dx_1}{(k-|x_1|)^{2s}} \notag \\ & \hspace{.5cm}+ \int_{\{|x_3|<1\}} \int_{\{|x_1|<k\}}dx_3 dx_1\int_{\{|x_2|<k\}}\frac{dx_2}{(k-|x_2|)^{2s}} \bigg) 
\leq  \frac{C kk^{1-2s}}{k^{2}}  \leq Ck^{-2s}.
\end{align}
Therefore, combining \eqref{example_plus_estimate1} and \eqref{example_plus_estimate2} we see that \eqref{example_plus_result} holds.

\end{example}

\begin{corollary}
\label{cor:domain condn}
Assume $\Omega$ is  Lipschitz set such that for some $c=c(\Omega)>0$,
$
   \mathcal{L}^n(\Omega\cap B_R)\geq cR^n\quad\text{for all $R>0$}. $
Then $P^1_{n,s}(\Omega)=0$ for $0<s<\frac{1}{2}$.

\end{corollary}

\begin{proof}
Take $U=B_1$ and note that $\dist(x,B_R^c)=R-|x|$ for $x\in \Omega\cap B_R$. Hence we get that
$$
    \frac{1}{\mathcal{L}^n(B_R\cap \Omega)}
    \int_{B_R\cap \Omega}\frac{dx}{\dist(x,B_R^c)^{2s}}
    \leq
    \frac{1}{cR^n}\int_0^R\frac{r^{n-1}}{(R-r)^{2s}}\leq \frac{C(n,s,\Omega)}{R^{2s}},
$$
which tends to $0$ as $R$ goes to $\infty.$ Therefore, the result follows by applying Theorem \ref{theorem:distance condition}.
\end{proof}
 
 \smallskip
    
\begin{example}
 Let $\Omega$ be the domain as in Example \ref{exmpl:domains with sufficient condn}(v) \textit{(Concentric Annulus)}. One can easily verify that  $\mathcal{L}^n(\Omega\cap B_R)\geq cR^n$ for any $R>0$. By Corollary \ref{cor:domain condn} we get $P^1_{n,s}(\Omega)=0$ for $0<s< \frac{1}{2}$. 
\end{example}

\subsection{Sufficient  conditions for  domains with finite measure }
\label{Section:sufficient condition for domains with finite measure}

\smallskip
 
  The main result of this section is Theorem \ref{theorem:general s smaller half in tube}, which is actually an application Theorem \ref{theorem:distance condition}, but we formulate it independently as many important class of domains come under it. For instance the  example  \eqref{example_not_rupert} of $\widetilde{\Omega}$ in the introduction: one can see  easily that $P_{n,s}^1(\widetilde{\Omega}) =0 $  for $s\in (0,\frac{1}{2})$ as an immediate application of Theorem \ref{theorem:general s smaller half in tube}.


\begin{definition}[\textbf{Decay condition in one direction}]
\label{def:decay nicely}
We say that $\Omega\subset\R^n$ satisfies \textit{decay condition in one direction} if for both $\Omega_+=\Omega$ and $\Omega_-:=\{(-x_1,x'):\,(x_1,x')\in\Omega\}$ the following holds: there exists some function $h:[0,\infty)\to [0,\infty)$ such that 
$$
    \mathcal{H}^{n-1}(\Omega_{\pm}\cap \{x_1=R\})
    \leq C h(R),\quad \forall\,R>0,\quad\text{ for some }C=C(\Omega)\geq 0,
$$
and there exists $a>0$ and an infinite sequence $\{R_k\}_{k\in\mathbb{N}}$ such that 
\begin{equation}
 \label{eq:nice decay at infty}
   \lim_{k\to\infty}R_k=\infty,\quad \lim_{k\to\infty}h(R_k)=0,\quad h(R_k+\eta)\leq h(R_k)
   \quad\text{for all $k$ and for all }
   \eta\in [0,a).
\end{equation}
\end{definition}

Here are some examples of sets which have finite measure and satisfy the decay condition in one direction.

\begin{example}
\label{examples to nice decay}
(i)
Let $f_1,f_2:\R\to \R$ be two Lipschitz functions, $ \Omega=\{x= (x_1,x_2) \in \R^2:\, f_1(x_1)\leq x_2\leq f_2(x_1)\}$ and assume 
 $\mathcal{L}^2(\Omega)<\infty.$ 
To see that $\Omega$ satisfies the decay condition in one direction, take $h=f_2-f_1$ and $R_k$ such that
$$
  \sup_{x\in [k,\infty)}h(x)=h(R_k).
$$
Using that $h$ is Lipschitz and $\mathcal{L}^n(\Omega)<\infty,$ it can be easily checked that this supremum is attained at some $R_k\in[k,\infty)$ and that $h(R_k)\to 0$ as $k\to\infty.$ Finally $a$ can be taken to be any positive real number.
 \smallskip
 
 \smallskip
 
 (ii) Let $\epsilon>0$ and for $\ell\in\mathbb{N}$, define $A(\ell):=(\ell,\ell+\frac{1}{\ell^{2+\epsilon}})\times (0,\ell)\subset\R^2.$ Set $\Omega=\bigcup_{\ell=1}^{\infty}A(\ell).$
 Here one can take $h$ as $\ell$ times the characteristic function
 $$
   h=\ell \chi_U,\quad\text{ where }U=\bigcup_{l=1}^{\infty}(\ell,\ell+\ell^{-2-\epsilon})
 $$
 and $a$ a positive number such that the distance between consecutive $A(\ell)$ is bigger than $a$ for all $\ell$ big enough.
\end{example}

\begin{theorem}
\label{theorem:general s smaller half in tube}
Let $0<s<\frac{1}{2}$ and $\Omega\subset\R^n$ be measurable Lipschitz set  and of finite measure $\mathcal{L}^n(\Omega)<\infty.$ Assume that $\Omega$ satisfies the decay condition in one direction and that for any $K>0$ the set $\Omega\cap\{|x_1|<K\}$ is bounded. Then 
$
P_{n,s}^1(\Omega)=0.
$
\end{theorem}

\begin{example}
The theorem applies to (i)-(ii) of Example \ref{examples to nice decay} and also to
 $\Omega=\{(x_1,x_2)\in\R^2:\,x_1>1,\;0<x_2<1/x_1^{1+\epsilon}\},$ where $\epsilon>0.$ 
\end{example}

\begin{remark}
It is immediate to see that the technique for the proof of the theorem, without modifications, shows that also for the domain, for $\epsilon>0,$
$$
   \Omega=\left\{|x_1|\geq 1,\,{x_2^2+x_3^2}<\frac{1}{|x_1|^{1+\epsilon}}\right\}
   \cup
   \left\{|x_2|\geq 1,\,{x_1^2+x_3^2}<\frac{1}{|x_2|^{1+\epsilon}}\right\}\cup  B_2(0)
   \subset\R^3,
$$
$P_{3,s}^1(\Omega)=0.$ Or any kind of domain where only a finite number of  ``tentacles" at infinity are unbounded in at most only one direction, which could be along a curve going to infinity, and that these tentacles satisfies the decay condition in one direction. 
An example of a domain to which the methods of Theorem \ref{theorem:general s smaller half in tube} would not apply is the set ($\epsilon >0$ so that $\Omega$ has finite measure).
$$
    \Omega=\left\{x\in\R^3:\,x_1^2+x_2^2>1,\,0<x_3<\frac{1}{(x_1^2+x_2^2)^{1+\epsilon}}\right\}.
$$
For another such example see Example \ref{example:union of 2 strips R3}.
\end{remark}

\smallskip

\begin{proof}[\textbf{Proof of Theorem \ref{theorem:general s smaller half in tube}} 
] 
  We first note that the property of decaying in one direction is not effected by cylindrical Schwarz symmetrization. Hence, by Lemma \ref{lemma:cylindrical Schwarz symm} it is sufficient to consider that $ \Omega $ is rotationally  symmetric about $ x_1 $ axis and the domain is parametrized by a function $f$, i.e. $\Omega= \{(x_1,x')\in \rn | \  |x'| < f(x_1)\}$.  We would like to show that there exist a bounded Lipschitz set $V\subset \rn$ and a sequence $\{\lambda_{k}\}$ tending to infinity as $k\rightarrow \infty$ such that 
\begin{align*}
\lim_{k\rightarrow\infty} \frac{1}{\lebm(\Omega\cap\lambda_kV)}\int_{\Omega\cap\lambda_kV}\frac{d x}{\dist(x,(\lambda_k V)^c)^{2s}}=0.
\end{align*}
We prove this result for dimension $2$ only to keep the argument relatively simple.  Choose $ V={(-1,1)}\times (-1,1)$ and ${\lambda_k}={R_k}+a$, where $R_k$ and $a$ are as in Definition \ref{def:decay nicely}.
As $\lebmtwo(\Omega)<\infty$, we get 
\begin{align}\label{small_leb_measure_outsise}
\lebmtwo \Big( \Omega \cap \big\{|x_1|\geq \frac{\lambda_k}{2}\big\}\Big) \rightarrow 0 \quad \mbox{as} \quad k\rightarrow \infty.
\end{align}
We write $\Omega\cap\lambda_k V=A_k \cup B_k \cup C_k,$ where,
\begin{align*}
A_k=& \{(x_1,x_2)\in \Omega :  |x_1|\leq \frac{\lambda_k}{2}, \ |x_2|\leq \lambda_k\},
\quad B_k= \{(x_1,x_2)\in \Omega : \frac{\lambda_k}{2}\leq |x_1|\leq R_k, \ |x_2|\leq \lambda_k\} ,\\
C_k=& \{(x_1,x_2)\in \Omega : R_k< |x_1|< \lambda_k, \ |x_2|\leq \lambda_k\} .
\end{align*}
We will show that 
\begin{align*}
\underbrace{\int_{A_k} \frac{d x}{\dist(x,(\lambda_k V)^c)^{2s}}}_{:=I_1} + \underbrace{\int_{B_k} \frac{d x}{\dist(x,(\lambda_k V)^c)^{2s}}}_{:=I_2} + \underbrace{\int_{C_k} \frac{d x}{\dist(x,(\lambda_k V)^c)^{2s}}}_{:=I_3}\longrightarrow 0 \quad \mbox{as} \quad k\rightarrow \infty.
\end{align*}
\smallskip

\noindent We first estimate $I_3$. By the assumption in Definition \ref{def:decay nicely} we get that $ \dist(x,(\lambda_k V)^c)=\lambda_k - |x_1|$ for any $x\in C_k$ whenever $\lambda_k$ is large enough.
Further, for large $\lambda_k$, by assumption \eqref{eq:nice decay at infty} we get 
\begin{align*}
 \int_{C_k} \frac{d x}{\dist(x,(\lambda_k V)^c)^{2s}}
 \leq  \int_{\{|x_2|< h(R_k)\}} \int_{\{R_k\leq |x_1|<\lambda_k\}} \frac{dx_2 d x_1}{(\lambda_k - |x_1|)^{2s}}
 \leq   K h(R_k) a^{1-2s}\rightarrow 0.
\end{align*}
  Next we estimate $ I_2 $, we split  $ B_k = E_k \cup F_k $ where
$E_k=  \{x\in B_k : \ |x_2|\leq |x_1|\} $ and 
$F_k= \{x\in B_k : \ |x_2|>|x_1|\}$. Therefore, we get that
\begin{align*}
\dist(x,(\lambda_k V)^c)=
\left\{
	\begin{array}{ll}
		\lambda_k - |x_1| \quad \mbox{if} \quad x \in E_k,\\
		\lambda_k - |x_2|  \quad \mbox{if} \quad x \in F_k.\\ 
		\end{array}
\right.
\end{align*}
If $ x \in E_k$ we get $\dist(x,(\lambda_k V)^c)\geq a,$ also note $\lebmtwo(E_k)\rightarrow 0$ as $k\rightarrow 0$ by \eqref{small_leb_measure_outsise}. Then
\begin{align}\label{estimate_E}
 \int_{x \in E_k}\frac{d x}{\dist(x,(\lambda_k V)^c)^{2s}}\leq \frac{\lebmtwo(E_k)}{a^{2s}}\rightarrow 0.
\end{align}
 If $x\in F_k$, we define, $Q_k= \Big\{x_1\in \mathbb{R}: \, \frac{{\lambda_k}}{2}\leq|x_1|\leq  R_k \quad \mbox{and} \quad  f(x_1)>|x_1|\Big\}.$ By definition of $Q_k$ we have 
\begin{align*}
  F_k=& \left\{ x\in \mathbb{R}^2:\, x_1\in Q_k,\quad  |x_1|<|x_2|<\min\{f(x_1),\lambda_k\}  \right\} \\\subset & 
\left\{ x\in \mathbb{R}^2:\, x_1\in Q_k,\quad  \frac{\lambda_k}{2}<|x_2|<\lambda_k\}  \right\}.
\end{align*}
  \noindent Then clearly ,
\begin{align}\label{estimate_F_first}
\int_{F_k} \frac{d x}{\dist(x,(\lambda_k V)^c)^{2s}}& \leq \int_{x_1\in Q_k}\int_{\{\frac{\lambda_k}{2}<|x_2|< \lambda_k\}} \frac{dx_1 \, dx_2}{{(\lambda_k - |x_2|)}^{2s}}  \leq K_2\lebmone(Q_k) \Big(\frac{\lambda_k}{2}\Big)^{1-2s},
\end{align}
where  $K_2$ is a constant independent of $k$. We also notice that
$Q_k \subset \big \{x_1\in \mathbb{R}: \, f(x_1)> \frac{\lambda_k}{2}\big \}$, then by Chebyshev 's inequality we see 
\begin{align}\label{bound_chebyshev_ineq_F}
 \lebmone(Q_k) \leq \lebmone \bigg\{ x_1 : \, f(x_1)>\frac{\lambda_k }{2}\bigg\}\leq \frac{\|f\|_{L^{1}(\mathbb{R})}}{\frac{\lambda_k}{2}} \leq \frac{2\lebmtwo(\Omega)}{\lambda_k}.
\end{align} 
 Hence, using \eqref{estimate_F_first} and \eqref{bound_chebyshev_ineq_F} we get
\begin{align}\label{estimate_F}
\int_{F_k} \frac{d x}{\dist(x,(\lambda_k V)^c)^{2s}}  & \leq \frac{K_2 \,\lebmtwo(\Omega)}{\lambda_k^{2s}}.
\end{align}
Combining \eqref{estimate_E} and \eqref{estimate_F} we get that $I_2 \rightarrow 0$ as $k \rightarrow \infty$.  Next, to estimate $I_1$ we define $ A_k:=M_k \cup N_k $, where
\begin{align*}
M_k=\Big\{x\in \Omega : \, |x_1|\leq \frac{\lambda_k}{2}, \ |x_2|\leq \frac{\lambda_k}{2}\Big\},\
N_k=\Big\{x\in \Omega : \, |x_1|\leq \frac{\lambda_k}{2},  \  \frac{\lambda_k}{2}<|x_2|< \lambda_k\Big\}.
\end{align*}
We get
\begin{align*}
\int_{M_k} \frac{dx}{\dist(x,(\lambda_k V)^c)^{2s}} \leq \frac{\lebmtwo(M_k)}{(\frac{\lambda_k}{2})^{2s}} \leq \frac{\lebmtwo(\Omega)}{(\frac{\lambda_k}{2})^{2s}} \rightarrow 0.
\end{align*}
Finally, the estimate on $N_k$ follows by similar argument  as in $F_k$. 
\end{proof}

\textbf{Acknowledgement:}    The second author was supported by the ERCIM ``Alain Bensoussan" Fellowship program and HRI postdoctoral fellowship grant. Parts of this work  were  carried out when the first and the second  author were visiting IIT, Kanpur. The hospitality  of IIT, Kanpur  is greatly acknowledged. The work of the third author  is supported  by Inspire grant IFA14-MA43. The third author  would  like to thanks the hospitality of Unicersitat Polyt\'ecnica  de Catalunya, where parts of this work is done. 
The authors would also like to thank Prof. E. Valdinoci, Prof. F. Rupert,  Prof. V. Mazya and Prof.  B. Dyda for  sharing their knowledge on the state of art of the problem at the initial stage of this research.


\begin{thebibliography}{200}
\small{

\bibitem{adams}R. A. Adams, and J. Fournier, Sobolev spaces, Second edition, Pure and Applied Mathematics(Amsterdam), Vol. 140(2003), Elsevier/Academic Press, xiv+305 pp.

\bibitem{Ambrosio Freddi Musina}
V. Ambrosio, L. Freddi, and R. Musina, Asymptotic analysis of the Dirichlet fractional Laplacian in domains becoming unbounded, \textit{J. Math. Anal. Appl.} 485 (2020), no. 2, 123845, 17 pp.


\bibitem{brasco} L. Brasco and A. Salort, A note on homogeneous  Sobolev space  of fractional order, preprint, \textit{arXiv:1806.08945} (2018).





\bibitem{picone} L. Brasco and G. Franzina, Convexity properties of Dirichlet integrals and Picone-type inequalities, \textit{Kodai Mathematical Journal,} 37(3)(2014), 769--799.




\bibitem{chipot1} M. Chipot, Asymptotic issues for some partial differential equations, \textit{Imperial College Press}, London, (2016).


\bibitem{delpino} M. Chipot, J. Davila and M. Del pino,  On the behavior of positive solutions of semilinear elliptic equations in asymptotically cylindrical domains, \textit{ J. Fixed Point Theory Appl.} 19 (2017), no. 1, 205--213.

\bibitem{mojsic} M. Chipot,  A. Mojsic and P. Roy,  On some variational problems set on domains tending to infinity,\textit{ Discrete Contin. Dyn. Syst.,} 36 (2016), no. 7, 3603--3621. 


\bibitem{crs} M. Chipot,  P. Roy and I.  Shafrir, Asymptotics of eigenstates of elliptic problems with mixed boundary data on domains tending to infinity, \textit{Asymptot. Anal., }85 (2013), no. 3--4, 199--227.

\bibitem{arn1} M. Chipot and A. Rougirel, On the asymptotic behaviour of the solution of elliptic problems in cylindrical domains becoming unbounded, \textit{  Commun. in Contemp. Math.,} 4(01) (2002) 15-44.

\bibitem{arn} M. Chipot and A. Rougirel, On the asymptotic behaviour of the eigenmodes for elliptic problems in domains becoming unbounded, \textit{ Trans. Amer. Math. Soc.,} 360 (2008), no. 7, 3579--3602.

\bibitem{Chen18} H. Chen, The Dirichlet elliptic problem involving regional fractional Laplacian, \textit{J. Math. Physics,} 59(7) (2018), 071504.



\bibitem{pi} I. Chowdhury and P. Roy, On the asymptotic analysis of problems involving fractional Laplacian in cylindrical domains tending to infinity, \textit{Commun. Contemp. Math.,} 19 (2017), no. 5, 21 pp.



\bibitem{ip} I. Chowdhury and P. Roy, Fractional Poincar\'e inequality for unbounded domains  with finite ball conditions: A Counter Example, (under preparation).




\bibitem{lk} B. Dyda, A fractional order Hardy inequality, \textit{Illinois J. Math.,} 48(2) (2004), 575--588.


\bibitem{dy}    B. Dyda and R. L. Frank, Fractional Hardy--Sobolev--Maz'ya inequality for domains, \textit{Studia Math.,} 208(2) (2012), 151--166.

\bibitem{vaha2}  B. Dyda, J. Lehrb\"ack, and A. V. V\"ah\"akangas,  Fractional Hardy-Sobolev type inequalities for half spaces and John domains. Proc. Amer. Math. Soc. 146 (2018), no. 8, 3393--3402.

\bibitem{dyda ihna Vaha}B. Dyda, L. Ihnatsyeva, A. V\"ah\"akangas,
On improved fractional Sobolev-Poincar\'e inequalities,
\textit{Ark. Mat.} 54 (2016), no. 2, 437--454. 


\bibitem{kass} M. Felsinger, M. Kassmann and P. Voigt, The Dirichlet problem for nonlocal operators, \textit{Mathematische Zeitschrift,} 279(3--4)(2015), 779--809.

\bibitem{valdi} A. Fiscella, R. Servadei and E. Valdinoci, Density properties for fractional Sobolev spaces, \textit{Ann. Acad. Sci. Fenn. Math.,} 40 (2015), no. 1, 235--253. 




\bibitem{fr}R. L. Frank, Eigenvalue Bounds for the Fractional Laplacian: A Review, preprint, \textit{arXiv:1603.09736} (2016).


\bibitem{rupert} R. L. Frank, T. Jin and  J. Xiong,  Minimizers for the fractional Sobolev inequality on domains,  \textit{Calc. Var. Partial Differential Equations,} 57 (2018), no. 2, Art. 43, 31.

\bibitem{Frank Seiringer} R. L. Frank and R. Seiringer, Non-linear ground state representations and sharp Hardy inequalities, \textit{J. Funct. Anal.,} 225 (2008), 3407--3430.

\bibitem{vaha1} R. Hurri-Syrjanen  and A. V. V\"ah\"akangas,  Fractional Sobolev-Poincare and fractional Hardy inequalities in unbounded John domains. \textit{Mathematika} 61 (2015), no. 2, 385--401.

\bibitem{vaha22} R. Hurri-Syrjanen  and A. V. V\"ah\"akangas, On fractional Poincaré inequalities. \textit{J. Anal. Math.}  120 (2013), 85--104. 

\bibitem{Li Wang} D. Li and K. Wang, Symmetric radial decreasing rearrangement can increase the fractional Gagliardo norm in domains, to appear in \textit{Comm. in Contemporary Math.,} (2018), 1850089.

\bibitem{Lions-Magenes} J. L. Lions and E. Magenes, Non homogeneous boundary value problems and applications,  \textit{Springer}, Volume 1, (1972).


\bibitem{loss} M. Loss and C. Sloane, Hardy inequalities for fractional integrals on general domains, \textit{J. Funct. Anal.,} 259 (2010), no. 6, 1369--1379.

\bibitem{mouhot} C. Mouhot, E. Russ, Y.Sire, Fractional Poincaré inequalities for general measures., \textit{J. Math. Pures Appl.,} 95 (2011), no. 1, 72--84.




\bibitem{Sandeep-Man} G. Mancini  and K. Sandeep, Moser-Trudinger inequality on conformal discs, \textit{Commun. Contemp. Math., } 12 (2010), no. 6, 1055--1068.

\bibitem{maz} V. G. Maz'ja, Sobolev Spaces, Springer Ser. Soviet Math., Springer, Berlin, (1985).

\bibitem{guide} D. Nezza, E. Palatucci and G. Valdinoci Enrico, Hitchhiker’s guide to the fractional Sobolev spaces,
\textit{Bull. Sci. Math.,} 136 (2012), no. 5, 521--573.



 

\bibitem{sar} R. Servadei and E. Valdinoci, The Brezis--Nirenberg result for the fractional Laplacian, \textit{Trans. Amer. Math. Soc.,} 367(1)(2015), 67--102.

\bibitem{raffela} R. Servadei and E. Valdinoci, Variational methods for non-local operators of elliptic type, \textit{Discrete Contin. Dyn. Syst.,} 33.5 (2013), 2105--2137.




\bibitem{strict positive} R. Servadei and E. Valdinoci, Weak and viscosity solutions of the fractional Laplace equation, \textit{Publicacions matemàtiques,} 58(1)(2014), 133--154.

\bibitem{karen} K. Yeressian, Asymptotic behavior of elliptic nonlocal equations set in cylinders, \textit{Asymptot. Anal.,} 89(1–2)(2014), 21--35.












}\end{thebibliography}
\end{document}